\documentclass{article}
\usepackage[utf8]{inputenc}
\usepackage[T1]{fontenc}
\usepackage[english]{babel}
\usepackage{amsmath,amssymb,amsfonts,amsthm}
\usepackage[inner=3.4cm,outer=3.65cm,bottom=3cm]{geometry}

\usepackage{hyperref}
\usepackage{color}
\definecolor{green}{RGB}{0,50,0}

\theoremstyle{remark}
\newtheorem{remark}{Remark}[section]
\theoremstyle{definition}
\newtheorem{theorem}{Theorem}[section]
\newtheorem{definition}[theorem]{Definition}

\newtheorem{proposition}[theorem]{Proposition}

\newtheorem{corollary}[theorem]{Corollary}
\newtheorem{hypothesis}[theorem]{Hypothesis}
\DeclareMathOperator{\R}{\mathbb{R}}
\DeclareMathOperator{\segn}{sign}

\DeclareMathOperator{\C}{\mathcal{C}}

\DeclareMathOperator{\ra}{\rightarrow}
\DeclareMathOperator{\de}{\text{d}}

\newcommand{\f}[1]{{\pmb{ #1}}}

\newcommand{\tet}{\tilde{\theta}}
\newcommand{\et}{\tilde{\eta}}

\newcommand{\tp}{\tilde{\varphi}}
\renewcommand{\t}{\partial_t}
\DeclareMathOperator{\essinf}{ess\,inf}
\DeclareMathOperator{\dive}{div}
\newcommand{\fhi}{\varphi}
\newcommand{\RR}{\mathbb{R}}
\newcommand{\ov}[1]{\overline{ #1}}
\newcommand{\un}[1]{\underline{ #1}}

\newenvironment{giuliorev}{\color{red}}{\color{black}}
\newcommand{\III}{\begin{giuliorev}}
\newcommand{\EEE}{\end{giuliorev}}

\newenvironment{bettirev}{\color{magenta}}{\color{black}}
\newcommand{\BBB}{\begin{bettirev}}
\newcommand{\EE}{\end{bettirev}}

\title{Weak solutions and weak-strong uniqueness for a thermodynamically consistent
phase-field model}

\author{Robert Lasarzik\\
Weierstrass Institute for Applied
Analysis and Stochastics,\\
Mohrenstrasse~39, 10117 Berlin,
Germany\\
E-mail: {\tt robert.lasarzik@wias-berlin.de}\\
\and 
Elisabetta Rocca\\
Dipartimento di Matematica, Universit\`a di Pavia, and IMATI - C.N.R.,\\
Via Ferrata~5, 27100 Pavia, Italy\\
E-mail: {\tt elisabetta.rocca@unipv.it}\\
\and
Giulio Schimperna\\
Dipartimento di Matematica, Universit\`a di Pavia, and IMATI - C.N.R.,\\
Via Ferrata~5, 27100 Pavia, Italy\\
E-mail: {\tt giusch04@unipv.it}
}
\date{}
\begin{document}

\maketitle

\begin{abstract}
In this paper we  prove the existence of weak
solutions for a thermodynamically consistent phase-field model introduced in \cite{fremond} 
in two and three dimensions of space. We use a notion of solution 
inspired by \cite{feireisl}, where the pointwise internal energy balance
is replaced by the total energy inequality complemented with
a weak form of the entropy inequality. Moreover, we prove 
existence of local-in-time strong solutions and, finally, we show weak-strong uniqueness of solutions,
meaning that every weak solution coincides with a local strong solution emanating 
from the same initial data, as long as the latter exists. 
\end{abstract}

\bigskip

\smallskip


\tableofcontents

\section{Introduction}
\label{sec:intro}

This paper is concerned with the analysis of the initial boundary-value problem for the following PDE system:
\begin{subequations}\label{eq}
\begin{align}
 \theta_t + \theta \varphi_t -\kappa   \Delta \theta  ={}& | \varphi_t|^2& & \text{in }\Omega\times (0,T)\,, \label{eq1}\\
 \varphi_t - \Delta \varphi + F' ( \varphi) ={}& \theta & & \text{in }\Omega\times (0,T)  \,,\label{eq2}
\end{align}
which describes phase transition phenomena
occurring in a bounded connected container~$\Omega\subset \R^d$, $d\in\{2,3\}$, with sufficiently 
smooth boundary and fixing a reference time interval $[0,T]$.
The state variables are the {\em absolute}\/ temperature $\theta$ and the order parameter 
(or phase-field) $\varphi$ describing the locally attained phase. 
We have denoted by $F$ the interaction potential entering the free energy functional and by $\kappa>0$ the 
heat conductivity, assumed to be constant.
This system is equipped with homogeneous \textsc{Neumann} 
boundary conditions,\textit{ i.e.,}
\begin{align}
\f n \cdot    \nabla \theta = 0 = \f n \cdot \nabla \varphi \qquad \text{on } \partial \Omega \times (0,T)\,\label{boundary}
\end{align}
and initial conditions
\begin{align}
\theta(0) = \theta_0 \,, \quad \varphi(0) = \varphi_0 \qquad \text{in } \Omega \,.\label{initial}
\end{align}
\end{subequations}
System \eqref{eq1}-\eqref{eq2} can be seen as one of the simplest 
diffuse-interface models describing non-isothermal phase transition 
processes in a thermodynamically consistent setting in the case when
no external heat source is present. More precisely, thermodynamic consistency
holds for a wide range of temperature
values and not only in proximity of the equilibrium temperature. 
Nonetheless, the global-in-time well posedness of the model
is still open both in 2 and in 3 space dimensions.

There are various ways to derive equations \eqref{eq1}-\eqref{eq2}
from the laws of Thermodynamicas. We sketch here an approach that 
follows the lines of the so-called \textsc{Fr\'emond} theory of phase 
transitions (\textit{cf}.~\cite[p.~5]{fremond}) for a particular choice of the free-energy
functional and of the pseudo-potential of dissipation. An alternative 
physical derivation is provided in the paper~\cite{BGCJV}. 

We start from the following expression for the volumetric free
energy:
\begin{equation}\label{psi}
\Psi(\theta,\varphi,\nabla\varphi)=c_V\theta(1-\log\theta)-\frac{\lambda}{\theta_c}(\theta-\theta_c)\varphi
+F(\varphi) +\frac{\nu }{2}|\nabla\varphi|^{2},
\end{equation}
where the constants
$c_V,\,\theta_c\,$, and $\nu>0$ represent, respectively, the specific heat,
the equilibrium temperature,
 and the interfacial energy coefficient, while  $\lambda$ stands for the latent
heat of the system. The term
$F(\varphi)+(\nu/2)|\nabla\varphi|^2$  accounts for a mixture or
interaction free-energy. Hereafter, for simplicity, we shall set
 $c_V=\nu=\lambda/\theta_c=1$ and we incorporate the term $\theta_c\varphi$
into $F(\varphi)$. 
Indeed, since in the following the potential $F$ will
be assumed to be a $\lambda$-convex function,
we may suppose without loss of generality that $\theta_c\varphi$
contributes to the non-convex part.
A typical example of a potential $F$ that we can include in our analysis is the so-called
``regular double-well potential'' $F(r)=(r^2-1)^2$.
Dissipation effects are described by means of a
pseudo-potential of dissipation $\Phi$ depending on the dissipative
variables $\nabla\theta$ and $\varphi_t$:
\begin{align}\label{fi}
\Phi (\nabla\theta, \varphi_t)=\,&
\,{\frac{1}{2}}{|\varphi_t|^2}+\frac{h(\theta)|\nabla\theta|^2}{2\theta}\,,
\end{align}
where  $h$ stands for a positive function representing the heat conductivity
of the process and, for the sake of simplicity, the other physical parameters have been
set equal to $1$.

The evolution of the phase variable $\varphi$
 is ruled by
an equation derived from a
generalization of the principle of virtual power
(\textit{cf}.~\cite[Sec.~2]{fremond}): 
\begin{equation}\label{mombal}
  B-\dive {\mathbf{H}}=0
\end{equation}
in case the volume
amount of mechanical energy provided to the domain by the external
actions (which do not involve macroscopic motions) is zero.
Here ${B}$ (a density or energy function) and
$\mathbf{H}$ (an energy flux vector) represent the internal
microscopic forces  responsible for the mechanically induced heat
sources:
\begin{align}
B=\frac{\partial \Psi }{
\partial\varphi}+\frac{\partial \Phi }{\partial\varphi_t}=-\theta
+F'(\varphi)+\varphi_t\,, \quad 
\label{constiH} \mathbf{H}=\frac{\partial
\Psi }{\partial \nabla\varphi}=\nabla\varphi.
\end{align}
With trivial computations, from \eqref{mombal}--\eqref{constiH} we
derive exactly \eqref{eq2}. Moreover, if  the surface amount of mechanical energy provided by
the external local surface actions (not involving macroscopic
motions) is zero as well, then the natural  boundary condition for
this equation of motion is exactly the second one in \eqref{boundary}. 

Finally, the energy balance equation reads  
\begin{equation}\label{enbal} 
  e_t+ \dive {\bf q}=B\varphi_t+{\bf
   H}\cdot\nabla\varphi_t\,,
\end{equation}
where $e$, the (specific) internal energy, is linked to the
 free energy $\Psi$ by the standard \textsc{Helmholtz} relation
\begin{equation}
\label{helmotz}
 e=\Psi+\theta s, \quad
s=-\frac{\partial \Psi}{\partial\theta},
\end{equation}
in which we have denoted by $s$ the specific entropy of the system. Following
\textsc{Fr\'emond}'s perspective (\textit{cf}.~\cite{fremond}), on the
right-hand side of~\eqref{enbal} there appears
the mechanically induced heat sources, related to
microscopic stresses, while the heat flux ${\bf q}$ is defined by the
following constitutive relation (\textit{cf}.~\eqref{fi}):
\begin{equation} \label{constiq}
{\bf q}=-\theta\frac{\partial\Phi}{\partial\nabla\theta}=-h(\theta)\nabla\theta.
\end{equation}
Using the no-flux boundary condition and the \textsc{Fourier} heat flux law (\textit{i.e.}~setting $h(\theta)\equiv \kappa$), we get the first boundary condition in \eqref{boundary} and equation \eqref{eq1}. 

This model turns out to be thermodynamically consistent in the sense that it complies with  the Second
Principle of Thermodynamics: indeed, the
\textsc{Clausius--Duhem} inequality
\begin{equation}\label{clausius-ineq} 
  s_t+ \dive \left(\frac{{\bf  q}}{\theta}\right)\geq 0
\end{equation}
holds true. To check
\eqref{clausius-ineq}, it is sufficient to note that
 the internal energy balance \eqref{enbal}
can be expressed in terms of the entropy $s$  in this way:
\begin{equation}\label{enteq}
  \theta\left(s_t+ \dive \left(\frac{{\bf q}}{\theta}\right)\right)
=|\varphi_t|^2 
-\frac{{\bf q}}{\theta}\cdot\nabla\theta = | \varphi_t|^2 + h(\theta ) \frac{| \nabla \theta|^2}{\theta},
\end{equation}
where we used~\eqref{constiq} and  the formal identity $e_t - B\varphi_t -  \mathbf{H} \cdot \nabla \varphi_t = \theta s_t- | \varphi_t|^2$, 
which follows from~\eqref{constiH} and~\eqref{helmotz}. 
Therefore, \eqref{clausius-ineq} ensues from the
positivity of $\theta$, a fact that we shall prove in the sequel.


Coming to our results, in this paper we shall first prove existence of weak solutions
$(\theta,\varphi)$. These will comply with equation \eqref{eq2} satisfied almost everywhere in the space-time 
domain $\Omega\times (0,T)$ and supplemented with homogeneous \textsc{Neumann} boundary conditions and initial conditions.
Moreover, weak solutions will satisfy the {\em total energy inequality}:
\begin{equation}\label{eneq-I}
E(t)\leq E(0), \quad\hbox{for a.e.~} t\in (0,T), \quad\hbox{where }\ E\equiv \int_\Omega\left(\theta+F(\varphi)
+\frac{1}{2}|\nabla\varphi|^2\right)\,\de\f x\,
\end{equation}
together with the following {\em entropy  inequality}:
\begin{align}\nonumber
  & - \int_{\Omega} \vartheta(t)  ( \log  \theta(t) + \varphi(t) ) \de \f x  
   + \int_{\Omega} \vartheta(0)  ( \log  \theta_0 + \varphi_0) \de \f x  
   + \int_0^t \int_{\Omega} \vartheta  
       \left ( \kappa | \nabla  \log \theta|^2+ \frac{| \varphi_t|^2}{\theta }  \right ) \de \f x  \de t\\
 \label{clausius-I}
  & \mbox{}~~~~~
   \leq \int_0^t \int_\Omega\left (  \kappa   \nabla \log \theta \cdot \nabla \vartheta 
       - \vartheta _t ( \log \theta + \varphi )\right ) \de \f x  \de t \,,
\end{align}
for a.e.~$t\in(0,T)$ and for every 
sufficiently regular nonnegative function $\vartheta$. 
%
%
Inequality \eqref{clausius-I} implies that the entropy is controlled by the
dissipation of the system. 
It is worth noticing that the weak formulation corresponds to the underlying physical laws, the two Thermodynamic Principles, 
\textit{i.e.,} energy conservation and entropy production in the case of~\eqref{eneq-I} and~\eqref{clausius-I}, respectively. 

From the mathematical viewpoint, initial-boundary value problems 
for equations \eqref{eq1}-\eqref{eq2} or variations of them have
been addressed in a number of 
apers. Starting from the pioneering work~\cite{bfl}, there is
 a comprehensive literature on the  models of phase
change with microscopic movements proposed by \textsc{Fr\'emond} (we may refer
to the PhD thesis~\cite{tesi-ulisse} and the references therein).

However, system \eqref{eq} needs to be carefully
handled, mainly because of the the presence of the terms $\varphi_t \,
\theta$ and $|\varphi_t|^2$ in \eqref{eq1}. Due to the difficulties arising from these two
higher order nonlinearities, there has not been any global-in-time
well-posedness result for the initial-boundary value problem related
to system \eqref{eq} {in}
the two or three-dimensional {case}. A global existence result for {system~\eqref{eq}}
has {only} been proved in the one-dimensional
setting in~\cite{ls1,ls2}, while in \cite{existence} (\textit{cf}.~also \cite{fpr-errata}) a global-in-time
well-posedness result has been obtained in the 2 and 3D cases only for power-like type growth of the 
heat flux law ($h(\theta)\sim \theta^\eta$, for $\eta\geq 3$ for $\theta$ large). 

Other approaches to phase transition models based on various forms of the {\sl entropy balance}
{and} possibly including mechanical effects are {available in the literature}.
A possibility (\textit{cf}.~\cite{bcf, bf, bfr, br}) 
consists in coupling an {\em entropy equation} (instead of the standard
internal energy balance equation \eqref{eq1}) with a microscopic
motion equation. Within this approach the resulting PDE system
couples an equation for $\varphi$ of the type \eqref{eq2} with an
entropy balance equation, which can be written as $s_t+ \dive 
\left(\frac{{\bf q}}{\theta}\right)=R$ and which is obtained
rewriting the internal energy balance in terms of $s$ by means of
the standard \textsc{Helmholtz} relation \eqref{helmotz} 
and assuming the right-hand side $R$, which now
has the meaning of an {\em entropy source}, to be known.
Note that in~\cite{bcfgnew} the entropy source is allowed to depend (somehow
singularly) on $\theta$. A second approach has been used in
\cite{roro1, roro2, roro3, roroweak}: the main novelty of these contributions lies in 
the fact that the equations for $\theta$ and $\varphi$ (analogous to our \eqref{eq1} and \eqref{eq2}) are coupled to a hyperbolic stress-strain
relation for the displacement variable $\mathbf{u}$. 
In the 3D case in \cite{roro1} a {\em local-in-time} well-posedness
result is obtained, whereas in \cite{roroweak} {\sl entropic solutions}\/ have been proved
to exist, but only in the case of a power-like type heat flux law ($h(\theta)\sim \theta^\kappa$, with $\kappa>1$ for $\theta$ large).
Finally, in \cite{roro2} the {\em global}
existence and the {\em long-time behavior of solutions} are
investigated in the 1D case. 

Since weak solutions of system \eqref{eq} are not known to be unique, and the source of non-uniqueness stems from
insufficient regularity properties holding in the setting of weak solutions,
a natural concept generalizing uniqueness is the so-called weak-strong uniqueness. 
It is fulfilled 
whenever every weak solution coincides with a local strong solution emanating 
from the same initial data, as long as the latter exists. In this way, the property also 
guarantees that the generalized solution concept is indeed a generalization of strong solutions. 

There are prominent examples in the context of fluid dynamics for these kinds of results, 
such as \textsc{Serrin}'s uniqueness 
result~\cite{serrin} for \textsc{Leray}'s weak solutions~\cite{leray} to the
incompressible \textsc{Navier}--\textsc{Stokes} equation in three space dimensions, 
or the weak-strong uniqueness for suitable weak-solutions to the 
incompressible \textsc{Navier}--\textsc{Stokes} system~\cite{Feireislrelative} 
or to the full \textsc{Navier--Stokes--Fourier} system~\cite{novotny}. 

A recurrent tool to prove such a weak-strong uniqueness result is the formulation of a relative 
energy. For  thermodynamical systems this idea goes back to \textsc{Dafermos}~\cite{dafermos}.
In the context of fluid dynamics, the relative energy approach has also been used to show the 
stability of a stationary solution~\cite{feireislstab}, 
the convergence to a singular limit~\cite{fei}, or to derive \textit{a posteriori} estimates
for simplified models~\cite{fischer}.

In the article at hand, this approach is adapted to a $\lambda$-convex energy functional. 
Actually, there are very few articles dealing with the relative 
energy approach for nonconvex energies, and all of them seem to pertain to
the context of liquid crystals. Actually, the paper~\cite{hyper} deals with a $Q$-tensor model 
equipped with a $\lambda$-convex energy and the authors are able to show weak-strong 
uniqueness for dissipative solutions. 
The weak-strong uniqueness for weak solutions to the penalized \textsc{Ericksen--Leslie} model 
in three space dimensions has been proved in~\cite{weakstrongweak}. 
In~\cite{weakstrong}, weak-strong uniqueness of weak solutions 
(measure-valued~\cite{masswertig} and dissipative~\cite{diss}) has been shown for the \textsc{Ericksen--Leslie}
model equipped with the \textsc{Oseen--Frank} energy, which is an energy with nonconvex leading order term.

Our proofs of existence and of weak-strong uniqueness for system~\eqref{eq1}-\eqref{eq2} combine
the use of more or less established methods in the mathematical theory of phase 
transition models with two new ideas, which constitute the key points of our 
argument:
\begin{itemize}
\item[-] a proper notion of weak solutions to \eqref{eq}, which is 
based on a new \textit{a-priori} estimate holding for
{\sl polynomial}\/ potentials $F$: actually,
the estimates following from the total energy inequality \eqref{eneq-I} and \eqref{clausius-I}
are not sufficient in order to pass to the limit in a suitable regularized problem;
\item[-] a concept of relative energy which combines the natural contribution of the ``physical''
energy with an additional $L^1$-term. Indeed, the latter is crucial in order to overcome the nonconvex
character of the physical energy functional and to obtain an effective estimate.
\end{itemize}

The plan of the paper is as follows: in the next Section~\ref{sec:main} we give our 
precise assumptions on data and state our main results. The remainder of the 
paper is devoted to proofs: in Section~\ref{sec:weak}, we show our global existence
result for weak solutions; in Section~\ref{sec:ws}, we deal with weak-strong uniqueness;
finally, in Section~\ref{sec:local}, we prove local-in-time existence for strong
solutions.


\section{Assumptions and main results}
\label{sec:main}

We start this section by presenting our basic hypotheses on the nonlinear 
function $F$. These assumptions are collected, together with a number 
of notable consequences of them,
in the following statement. It is worth noting that, despite the
length of what follows, it is very easy to check that
the usual double-well potentials of polynomial growth
(as the commonly used quartic potential $F(r)=(r^2-1)^2$) satisfy 
all the assumptions (A)-(D) listed below.
\begin{hypothesis}
\label{hypo}
 (A)~~We let $ F \in \C^2(\R, \R) \cap \C^{2,1}_{\text{loc}}( \R,\R)$. 

\smallskip
\noindent%
(B)~~We assume $F$ to be {\it $\lambda$-convex}, \textit{i.e.}, convex up to a quadratic 
perturbation. Namely there exists a constant $\lambda>0$ such that 
$ F''(y) \geq - \lambda$ for all $ y \in \R$.
We can then define a convex modification of $F$, subsequently named $G$, as 
\begin{align}
  G(y ) = F( y) + \lambda y^2 \, \quad y \in \R\label{lambdacon}.
\end{align}
By construction, $G$ is ``strongly convex'', \textit{i.e.}, $G''(y)\ge \lambda >0$
for all $ y \in \R$. Moreover, it is not restrictive to assume $G$ to 
be nonnegative and so normalized that $G'(0)=0$.

\smallskip
\noindent%
(C)~~Next, we assume a minimal coercivity assumption at $\infty$,
namely 
\begin{equation}\label{Fcoerc}
  \liminf_{|y|\ra \infty} F'(y) \segn y > 0.
\end{equation} 
As a consequence of \eqref{Fcoerc}, we can first observe that 
$F(y) \geq -c $ for  some constant $c>0$ and every $y\in \R$.
Moreover, it is easy to verify that
the physical energy controls the $H^1$-norm of $\fhi$ {from above}. Namely, 
there exist $\gamma > 0$ and $c\ge 0$ such that
\begin{equation}\label{Ecoerc}
  \frac12 \| \nabla \fhi \|_{L^2(\Omega)}^2
   + \int_\Omega F(\fhi) \, \de \f x \ge \gamma \| \fhi \|_{H^1(\Omega)} - c,
\end{equation} 
for every $\fhi \in H^1(\Omega)$ such that $F(\fhi) \in L^1(\Omega)$.

\smallskip
\noindent%
(D)~~Finally, a growth condition is assumed to hold, \textit{i.e.}, 
there exists a constant $c>0$ such that 
\begin{equation}\label{growthF}
  | F' (y) |\log (e + | F'(y) |) \leq c( 1 + |F(y)| ) 
   \quad\text{for all }y \in \R.
\end{equation} 
Possibly modifying the value of $c$ one can see that the
analogue of \eqref{growthF} holds also for the convex 
modification~$G$, \textit{i.e.}, we have
\begin{equation}\label{growthG}
  | G' (y) |\log (e + | G'(y) |) \leq c( 1 + G(y) ) 
   \quad\text{for all }y \in \R.
\end{equation} 
To check that \eqref{growthF} implies \eqref{growthG},
a number of straightforward but somehow technical
computations would be required. We leave them to the reader 
because no real difficulty is involved. 
\end{hypothesis}

We can now define weak solutions to our system in a rigorous way:
\begin{definition}\label{def:weak}
A couple $ (\theta, \varphi) $ is called a weak solution to~\eqref{eq}
over the time interval $(0,T)$ if the following conditions are satisfied. First, 
there hold the regularity properties
\begin{subequations}
\begin{align}\label{rego:t1}
  \theta & \in L^\infty( 0,T; L^1(\Omega)) 
   \quad \text{with } \theta(\f x ,t) > 0 \text{ a.e.~in }\Omega \times (0,T) \,, \\ 
 \label{rego:t1b}
  \theta \log \theta & \in L^1(\Omega \times (0,T)) \,,\\
 \label{rego:t2}
  \log \theta & \in L^\infty(0,T;L^1(\Omega)) \cap L^2(0,T; H^1(\Omega)),\\ 
 \label{rego:t2b}
  \t \log \theta & \in (\mathcal{M}(\ov\Omega \times [0,T] ) + L^2(0,T;(W^{1,2}(\Omega)^*)) \,,\\
 \label{rego:fhi}
  \varphi & \in L^\infty (0,T; H^1(\Omega))\cap  W^{1,1}(0,T;L^1(\Omega)) \,,\\
 \label{rego:deltafhi}
  \Delta \varphi & \in L^1(0,T;L^1(\Omega)) \,,\\
 \label{rego:Ffhi}
  F(\varphi) & \in L^\infty(0,T; L^1(\Omega))  \,, \\
 \label{rego:tf}
  \theta^{-1/2}\t \varphi  & \in L^2(0,T; L^2(\Omega)) \,.
\end{align}
\end{subequations}
Next, the entropy inequality holds in the integral form
\begin{align}\nonumber
  &  - \int_{\Omega} \vartheta(t)  ( \log  \theta(t) + \varphi(t) ) \de \f x  
   + \int_{\Omega} \vartheta(0)  ( \log  \theta_0 + \varphi_0) \de \f x\\
 \nonumber
    & \mbox{}~~~~~~~~~~~~~~~
   + \int_0^t \int_{\Omega} \vartheta  
       \left ( \kappa | \nabla  \log \theta|^2 + | \theta^{-1/2}\t \varphi|^2 \right ) \de \f x  \de t\\
 \label{entropy}
  & \mbox{}~~~~~
   \leq \int_0^t \int_\Omega\left (  \kappa   \nabla \log \theta \cdot \nabla \vartheta 
       -\t \vartheta  ( \log \theta + \varphi )\right ) \de \f x  \de t \,,
\end{align}
for a.e.~$t\in(0,T)$ and for every $\vartheta \in C^0([0,T]\times \overline\Omega) 
\cap H^1(0,T; L^2(\Omega)) \cap L^2 (0,T; H^1(\Omega))$ such that
$\vartheta(t, \f x)\ge 0$ for every $t\in[0,T]$ and $\f x \in \overline\Omega$.
Moreover, the phase field equation 
\begin{align}
  \t \varphi - \Delta \varphi + F'(\varphi ) = \theta   
    \label{phaseeq}
\end{align}
holds a.e.~in $\Omega \times (0,T)$ with the initial and boundary
conditions
\begin{equation}\label{ibfhi}
  \varphi(0) = \varphi_0, \qquad
   \f n \cdot \nabla \varphi = 0
\end{equation} 
in the sense of traces respectively in $\Omega$ and on $\partial \Omega \times (0,T)$.
Finally, we require validity of the energy inequality for a.e.~$t\in (0,T)$:
\begin{align}
  \int_{\Omega } \left (\frac12 | \nabla \varphi(t)|^2 + F( \varphi(t)) + \theta(t) \right )\de \f x  
    \leq \int_{\Omega } \left ( \frac12 | \nabla \varphi_0|^2 + F( \varphi_0) + \theta_0\right ) \de \f x 
  \,.\label{energyin}
\end{align}
\end{definition}
\begin{remark}
Note that the traces are well-defined a.e.~in $\Omega$ and on $\partial \Omega \times (0,T)$, respectively. 
Indeed, from~\eqref{rego:fhi} we observe that $\varphi\in \C_{\text{w}}([0,T];H^1(\Omega))$. Moreover, 
 the normal-trace operator 
 is well-defined as a mapping from $W^{2,1}(\Omega)$ to $L^1(\partial \Omega)$, see for instance~\cite[Thm.~2.7.4, (2.7.10) with (2.7.4)]{brezzi},~\cite[Prop.~3.80]{demengel}, or~\cite{normaltrace}. 
\end{remark}
The first result of this paper is devoted to proving the global in time existence of weak solutions
in the sense of Definition~\ref{def:weak}.
As noted in the introduction, this seems to be the first rigorous existence result 
for system \eqref{eq1}-\eqref{eq2} in absence of regularizing power-like terms in the heat equation (\textit{cf}.~\cite{fpr-errata}).
\begin{theorem}\label{thm:exweak}
 Let $\Omega$ be sufficiently smooth and let Hypothesis~\ref{hypo} be fulfilled. Let us also assume
 \begin{subequations}
 \begin{align}\label{in:teta}
  \theta_0 & \in L^1(\Omega), \quad
   \theta_0 > 0~~\text{a.e.~in~$\Omega$}, \quad
   \log\theta_0 \in L^1(\Omega)\,,\\
  \label{in:chi}
   \fhi_0 & \in H^1(\Omega), \quad F(\fhi_0)\in L^1(\Omega), \quad 
   \varphi_0 \log \theta_0 \in L^1(\Omega)\,,\\
  \label{in:chi2}
  \fhi_0(\f x) & \ge - K > - \infty \quad\text{for some }
   K \ge 0~~\text{and a.e.}\ \f x \in \Omega\,. 
 \end{align} 
 \end{subequations}
 Then, there exists at least one weak solution in the sense of Definition~\ref{def:weak}.
 Moreover, if we have in addition
 \begin{equation}\label{tetapos}
   \essinf_{\f x \in \Omega} \theta_0(\f x) > 0\,,
 \end{equation}
 then there follows the minimum principle property
 \begin{equation}\label{posifo}
   \essinf_{\Omega\times(0,T)} \theta(\f x ,t) > 0\,,
 \end{equation}
 providing also the additional regularity
 \begin{equation}\label{extra-log}
   \log \theta \in \text{BV}(0,T;(W^{1,p}(\Omega))^*)\cap L^\infty(0,T;L^q(\Omega)) \quad\text{for }p>3 \text{ and }q\in (1,\infty)\,,
 \end{equation}
 so that in particular the entropy inequality~\eqref{entropy} holds for every (and not just a.e.)~$t\in[0,T]$. 
\end{theorem}
Observe that \eqref{extra-log} yields that there exists an at most countable
set $D\subset[0,T]$, such that $\log\theta\in C^0([0,T] \setminus D;(W^{1,p}(\Omega))^*)$.

The second result of this paper states the local-in-time existence of strong solutions. 
Here and below we note as $H^2_{\f n}(\Omega)$ the space of the functions in $H^2(\Omega)$
having zero normal derivative on the boundary. 
We also need an additional assumption on $F$:
\begin{hypothesis} \label{hypo2}
  Let $F$ be three times continuously differentiable, \textit{i.e.,} $ F \in \C^3(\R;\R)$. 
\end{hypothesis}
\begin{theorem} \label{thm:local}
 Let Hypotheses~\ref{hypo} and \ref{hypo2} hold true. Moreover, let us assume the following 
 additional conditions on the initial data:
 \begin{subequations}
 \begin{align}\label{teta0}
   & \theta_0 \in H^1(\Omega), \quad \essinf_{\f x \in \Omega} \theta_0(\f x) > 0\,,\\
  \label{chi0}
   & \varphi_0\in H^2_{\f n}(\Omega), \qquad \varphi_1:= \Delta \varphi_0 - F'(\varphi_0) + \theta_0 \in H^1(\Omega)
 \end{align} 
 \end{subequations}
 (notice that the latter is equivalent to just assuming $\varphi_0\in H^3(\Omega)\cap H^2_{\f n}(\Omega)$). 
 Then, there exists a local-in-time solution to the system~\eqref{eq}, 
 namely, a couple $(\theta,\varphi)$ satisfying the 
 regularity properties
 \begin{subequations}
  \begin{align}\label{regotetas}
    & \theta \in \C^0([0,T^*);H^1(\Omega)) \cap L^2(0,T^*;H^2(\Omega)) \cap W^{1,2}(0,T^*;L^2(\Omega))\,,\\
  \label{regochis}
    & \varphi \in \C^0([0,T^*);H^2(\Omega))  \cap C^1([0,T^*);H^1(\Omega)) \cap W^{1,2}(0,T^*;H^2(\Omega))\,, 
 \end{align}
 \end{subequations}
 for a certain $T^*>0$. Moreover,  the analogue of \eqref{posifo} holds on the 
 interval $[0,T^*)$ 
 and the strong solution fulfills a supplementary initial condition in the sense 
 that $\| \varphi_t(t)-\varphi_1\|_{H^1}\to 0$ as $ t\searrow 0$. 
 Finally, the strong solution satisfies the energy equality \eqref{energyin} with equality sign for any $t\in [0,T^*)$.
\end{theorem}
Actually, the regularity proved above for strong solution is not sufficient
for our purposes. However, it is easy to prove additional results. For instance 
we have the following
\begin{corollary} \label{thm:local2}
 Let the assumptions of Theorem~\ref{thm:local} hold and let, in addition, 
 \begin{equation}\label{teta1}
    \theta_0 \in H^2_{\f n}(\Omega).
 \end{equation} 
 Then, the local strong solution provided by Theorem~\ref{thm:local} satisfies the 
 additional regularity condition
 \begin{equation}\label{regotetas2}
    \theta \in \C^0([0,T^*);H^2(\Omega)) \cap W^{1,2}(0,T^*;H^1(\Omega)).
 \end{equation} 
\end{corollary}
\begin{proof}
It is sufficient to observe that, 
thanks to~\eqref{regotetas}-\eqref{regochis},
\eqref{eq1} can be rewritten as
\begin{equation}\label{eq1b}
  \theta_t - \kappa \Delta \theta  = (\t \fhi)^2 - \theta \t \fhi \in L^2(0,T^*;H^1(\Omega)).
\end{equation} 
Hence, the assertion follows from the standard theory of linear parabolic equations 
by also using the regularity \eqref{teta1} of the initial datum (the a-priori estimate
used here corresponds to testing \eqref{eq1b} by $-\Delta\theta_t$ and using the regularity of 
the right-hand side resulting from \eqref{regotetas}--\eqref{regochis}).
\end{proof}
\begin{remark}\label{rem:continuous}
Via  standard bootstrap arguments, one can even deduce more regularity for the local-strong solution, \textit{e.g.}~that all 
terms appearing in~\eqref{eq} are even continuous in the case that $ \theta_0,\varphi_0 \in \C^2(\ov\Omega)$. 
Indeed, from $\varphi$, $\theta \in \C^0([0,T^*); H^2(\Omega))$ we infer by the \textsc{Sobolev}-embedding theorem 
that $\varphi$, $\theta \in \C^0(\ov \Omega \times [ 0,T^*))$. 
The continuity of the function $F'$ implies that 
\begin{align*}
\t \varphi - \Delta \varphi = \theta - F'(\varphi) \in C^0(\ov \Omega \times [0,T^*))\,,
\end{align*}
which grants together with the   regularity of the heat equation (see~\cite[Prop.~1.1A]{taylor}) that there exists 
some $0<\tilde{T}<T^*$ such that $\varphi \in \C^1([0,\tilde{T}];\C^0(\ov\Omega))\cap \C^0([0,\tilde{T}];\C^2(\ov\Omega))$, \textit{i.e.},
$\t \varphi$, $\Delta \varphi\in \C^0(\ov\Omega\times [0,\tilde{T}])$. 
This in turn implies that the right-hand side of~\eqref{eq1b} is even in~$C^0(\ov \Omega \times [0, \tilde{T}])$ such that, we may infer 
the regularity $\theta \in \C^1([0,\tilde{T}];\C^0(\ov\Omega))\cap \C^0([0,\tilde{T}];\C^2(\ov\Omega))$,
(see~\cite[Prop.~1.1A]{taylor}), \textit{i.e.} $\t \theta$, $\Delta \theta\in C^0(\ov \Omega \times [0,\tilde{T}])$. 
This regularity is then sufficient to give a sense to all terms appearing in the weak-strong uniqueness proof (see Sec.~\ref{sec:ws}) 
and especially to bound the right-hand side of the inequality~\eqref{estimate}. 
It is worth noticing that this regularity is exactly what is needed to fulfill the system~\eqref{eq} in a classical sense, \textit{i.e.}, pointwise.
\end{remark}

\smallskip

We can finally state our second main result providing weak-strong uniqueness for system
\eqref{eq1}-\eqref{eq2}. We may note that the proof of this result neither
requires Hypothesis~\ref{hypo2}, nor any additional assumption on $F$ and on the initial data
leading to the validity of \eqref{eq1}-\eqref{eq2} in $L^\infty$ according
to the above remark. On the other hand, such additional conditions are needed
in the construction of local strong solutions with the desired properties
(see Theorem~\ref{thm:local} above). 
\begin{theorem}\label{thm:main}
 Let Hypothesis~\ref{hypo} hold true. Let $( \theta , \varphi)$ be a weak solution according 
 to Definition~\ref{def:weak} and $( \tet ,\tp)$ a strong solution according to
 Theorem~\ref{thm:local} and satisfying\/ \eqref{eq1}--\eqref{eq2} in
 the $\C^0$-sense.
 Let us also assume that $( \theta , \varphi)$ and $( \tet ,\tp)$ 
 start from the same initial data assumed to fulfill\/ \eqref{teta0}-\eqref{chi0} and \eqref{teta1}
 and that they are defined over the same reference interval $(0,T)$.
 Then the two solutions coincide,\textit{ i.e.}, $ \theta \equiv \tet $ and $\varphi \equiv \tp $
 a.e.~in~$\Omega\times(0,T)$.
\end{theorem}
It is worth observing that, in view of \eqref{teta1}, the strong solution additionally
satisfies \eqref{regotetas2}; moreover, thanks to the last condition in \eqref{teta0},
both the weak and the strong solution comply with~\eqref{posifo}.
\begin{remark}
The two main theorems of this article, \textit{i.e.,} Theorem~\ref{thm:exweak} and Theorem~\ref{thm:main}, 
are not restricted to the space dimensions two and three. Actually, the proofs of these results do not use 
any dimension-depending inequality. Only Theorem~\ref{thm:local} makes explicit use of the restriction onto at most three space dimensions, 
since several dimension-dependent embeddings are crucial. Nevertheless, we expect that the local existence might be 
generalized to higher space dimensions up to technical variants in the proof.
\end{remark}

\smallskip

The rest of the paper is devoted to the proof of the above results,
along the scheme outlined at the end of the introduction.


\section{Existence of weak solutions}
\label{sec:weak}

In this section, we prove the existence of weak solutions according 
to Definition~\ref{def:weak}.


\subsection{Regularized system}

As a regularized system to approximate~\eqref{eq}, we consider
\begin{subequations}\label{sys:reg}
\begin{align}
\t \theta 
- \kappa \Delta \theta 
+ \varepsilon \theta^{p}  + \theta \t\varphi  ={}& |\t \varphi|^2 \,,\label{sys:energy} \\
\t \varphi - \Delta \varphi + F'(\varphi ) ={}&\theta \,,\label{sys:phase}
\end{align}
where $\varepsilon\in(0,1)$ is intended to pass to $0$ in the limit and 
$p\in (3,\infty)$ is arbitrary but fixed.
The system is equipped with the boundary conditions
\begin{align*}
 \f n \cdot
  \nabla \varphi = 0 = \f n \cdot \nabla \theta \,.
\end{align*}
\end{subequations}
In order to prove existence of weak solutions according to Definition~\ref{def:weak},
we follow the standard path of first proving existence of solutions to the regularized 
system~\eqref{sys:reg} (see Subsection~\ref{sec:exreg}) (complemented with
regularized initial data), then deriving \textit{a priori} estimates 
(see Subsection~\ref{sec:apri}) uniform with respect to~$\varepsilon$, 
and showing convergence in the end (see Subsection~\ref{sec:conv}).


\subsection{Existence of solutions to the regularized system\label{sec:exreg}}

We consider the following regularized decoupled system 
\begin{subequations}
\label{app:sys}
\begin{align}
  \t \theta -\kappa \Delta  \theta + \varepsilon (\theta)^p + \theta \t \varphi ={}& |\t  \varphi|^2 \,,\label{app:heat} \\
\t \varphi - \Delta \varphi + F'(\varphi ) ={}&\ov \theta \,,\label{app:phase}
\end{align}
\end{subequations}
equipped with the boundary conditions~\eqref{boundary} and initial data $ \theta_ 0^\varepsilon \in \C^2(\ov \Omega)$  with $\theta_0^\varepsilon (\f x) \geq c_\varepsilon>0$ 
and $ \varphi_0^\varepsilon \in   \C^2(\ov \Omega )$ for every $\varepsilon>0$ 
such that $\theta_ 0^\varepsilon$ and  $\varphi_0^\varepsilon $ strongly converge to $ \theta_0$ 
and $\varphi_0$ in the spaces indicated in~\eqref{in:teta} and~\eqref{in:chi}, respectively.
Additionally, we assume that $\varepsilon \| \varphi_0\|_{L^p(\Omega)} ^p \leq c$
and that the product $\varphi^\varepsilon _0 \log \theta^\varepsilon_0$ remains bounded in $L^1(\Omega)$.
Here and below, for $r\in\R$ and $\sigma>1$, we use the notation 
$(r)^\sigma=|r|^{\sigma-1}r$. 
This assures that the associated operator 
is monotone. This is necessary in this step, because we have not 
established yet that $\theta $ is positive.

We prove the existence of local solutions by means of the \textsc{Schauder} fixed point argument. 
To do this, we define the set
\begin{align*}
\mathcal{M}_R:= \big \{ \theta \in L^p (\Omega \times (0,T_\varepsilon) )
 \big | \| \theta \|_{L^p(\Omega \times (0,T_\varepsilon))} \leq R  \big \}, 
\end{align*}
where the radius $R>0$, possibly depending on $\varepsilon$, and the final time 
$T_\varepsilon>0$ will be specified in the course of the procedure.

The operator $ \mathcal{T}: \mathcal{ M} _R \ra \mathcal{M}_R $ is defined
by the following procedure. For given $\ov \theta \in \mathcal{M}_R$, 
solve~\eqref{app:phase} with the boundary condition~\eqref{boundary} 
and the initial value~$\varphi_0^\varepsilon  $ so obtaining $\varphi$ 
in a proper regularity class.

Then, for given $\t \varphi $, solve~\eqref{app:heat} with the boundary 
condition~\eqref{boundary} and the initial value~$\theta_0^\varepsilon$ obtaining ${\theta}$. 
Correspondingly, we set $\mathcal{T}(\ov \theta) = {\theta}$. 
Then, to be able to apply the \textsc{Schauder} fixed point theorem, we need to prove several claims.

\smallskip

\textbf{Claim 1: The operator $\mathcal{T}$ is well defined and maps $\mathcal{M}_R $ into itself.}
Since showing the existence of weak solutions to the two equations separately is a standard procedure
(one may use \textsc{Rothe}'s method or a suitable \textsc{Galerkin} discretization like in~\cite[chapter~8]{roubicek}),
we rather focus on the relevant estimates. 

First of all, recalling~\eqref{lambdacon} and testing~\eqref{app:phase} 
by $(\varphi)^{p-1} /p$, we find
\begin{multline*}
\frac{\de }{\de t} \| \varphi\|_{L^p(\Omega)}^p + \frac{4(p-1)}{p^3 } \| \nabla ( \varphi )^{p/2}\|_{L^2(\Omega)}^2 
 +\frac{1}{p} \int_\Omega  G'(\varphi)(\varphi)^{p-1} \de \f x \\ 
  \leq \frac{2\lambda}p \|\varphi\|_{L^p(\Omega)}^p 
  + \frac{1}{p^2} \|  \bar \theta \|_{L^p(\Omega)}^p + \frac{p-1}{p^2} \| \varphi \|_{L^p(\Omega)}^p \,.
\end{multline*}
Due to the monotonicity of $G'$ and to the assumption $G'(0)=0$, the \textsc{Gronwall} inequality shows that 
\begin{equation}
  \| \varphi \|_{L^\infty(0,T_\varepsilon;L^p(\Omega))} 
   \leq   \| \varphi_0^\varepsilon \|_{L^p(\Omega)} + c \| \bar \theta \|_{L^p(0,T_\varepsilon;L^p(\Omega))}
   \leq  \| \varphi_0^\varepsilon \|_{L^p(\Omega)} + c R \,.\label{phiLp}
\end{equation}
Here and below, the constant $c$, whose value may vary on occurrence,
is allowed to depend both on $\varepsilon$ and, {\it in a monotonic way},
on the final time $T_\varepsilon$ to be chosen later. 

Let us now set
\begin{equation}\label{defiGp}
  G_p(s):= \int_0^s (G'(r))^{p-1}  \de r
\end{equation} 
and note that, by construction, $G_p$ is convex and coercive at infinity.
Then, testing~\eqref{app:phase} by $ (G'(\varphi)) ^{p-1}  $, we observe
by \textsc{Young}'s inequality that
\begin{align*}
  \frac{\de }{\de t} \int_\Omega G_p(\fhi) \de x 
  + (p-1) \int_\Omega |\nabla  \varphi|^2 G''(\varphi ) | G'(\varphi )|^{p-2} \de \f x 
  +  \| G'(\varphi ) \|_{L^p(\Omega)}^p \\ 
  \leq \frac{1}{2} \| G'(\varphi ) \|_{L^p(\Omega)}^p 
  + c\left  ( \| \bar \theta \|_{L^p(\Omega)}^p + \| \varphi \|_{L^p(\Omega)}^p\right  ) \,.
\end{align*}
We then obtain
\begin{equation*}
  \| G' (\varphi)\|_{L^p(\Omega\times (0,T_\varepsilon))}
  \leq \int_\Omega G_p(\fhi_0^\varepsilon) \de \f x
   + c \left (\| \bar \theta \|_{L^p(\Omega\times (0,T_\varepsilon))}
     + \| \varphi \|_{L^p(\Omega \times (0,T_\varepsilon))} + 1 \right )\,. 
\end{equation*}
Comparing terms in \eqref{app:phase} and using
the maximal $L^p$-regularity of the operator associated to the heat equation (see~\cite{book}), 
we conclude that there exists a constant $c>0$ such that  
\begin{multline}
   \| \t \varphi \|_{L^p(\Omega \times (0,T_\varepsilon))}
  + \| \Delta \varphi \|_{L^p(\Omega \times (0,T_\varepsilon))} \\
  \leq  c( \| \varphi_0^\varepsilon\|_{\C^0(\ov\Omega)} 
  + \| \varphi_0^\varepsilon\|_{W^{2-2/p,p}(\Omega)} 
   +  \| \bar \theta \|_{L^p(\Omega \times (0,T_\varepsilon))}) \leq c(1+ R)\, . 
\label{esttimeapp}
\end{multline}
In the following, we may consider equation~\eqref{app:heat} for $\t \varphi\in L^p(\Omega \times (0,T))$. 
We define $ q = p^2 /2$ and test equation~\eqref{app:heat} by $\theta^{(q-p)}$ to find
\begin{align*}
 \frac{1}{(q-p+1)} \frac{\de }{\de t } \| \theta \|_{L^{(q-p+1)}(\Omega)} ^{q-p+1} 
  + \frac{4 \kappa (q-p)}{(q-p+1)^2} \| \nabla (\theta ^{(q-p+1)/2}) \|^2_{L^2(\Omega)} 
  +\varepsilon  \| \theta \|_{L^{q}(\Omega)}^{q} \\
  = \int_\Omega |\t \varphi |^2 \theta ^{(q-p)} - \t \varphi  \theta ^{(q-p+1)} \de \f x \\ 
  \leq \frac{\varepsilon}{2} \| \theta \|_{L^{q}(\Omega)}^{q} 
  + c_\varepsilon \left (\| \t \varphi \|_{L^p(\Omega)}^p
  +\| \t \varphi \|_{L^{q/(p-1)}(\Omega)}^{q/(p-1)} \right )\,.
\end{align*}  
Note that $p\geq 2$ such that $  {q/(p-1)}\leq p$. Integrating in time, we conclude that  
\begin{align}
    \| \theta \|_{L^{p^2/2}(\Omega \times (0,T_\varepsilon))}^{p^2/2} 
  \leq  c_\varepsilon \left ( \| \theta_0^\varepsilon \|_{L^{(p^2-2p+2)/2}(\Omega)}^{(p^2-2p+2)/2}    
  + \| \t \varphi \|_{L^p(\Omega\times (0,T_\varepsilon))}^p\right )\,.\label{estimthetapsq}
\end{align}    
The previous estimates and maximal $L^{p/2}$-regularity results 
for the operator associated to the heat equation allow us to infer that 
\begin{align}
  \| \t \theta \|_{L^{p/2}(\Omega\times (0,T_\varepsilon))}
  + \| \Delta  \theta \|_{L^{p/2}(\Omega\times (0,T_\varepsilon))} \leq c \,.\label{maxregtheta}
\end{align}
Using interpolation, it is now not difficult to verify (we leave details to the reader)
that there exists an $s>0$ such that 
\begin{align*}
 \| \theta \|_{W^{s,2p}(0,T_\varepsilon;L^{p}(\Omega))} \leq 
  Q ( \| \varphi_0^\varepsilon\|_{ \C^2(\ov\Omega) } , 
   \| \theta_0^\varepsilon \|_{\C^2(\ov\Omega)}, R)\,,
\end{align*}
where $Q$ is a computable nonnegative-valued and continuous function, which is monotone in $R$.
As the previous constants $c$, also the expression of $Q$ may depend on $T_\varepsilon$. 
However, this dependence is monotonic, hence taking a small final time $T_\varepsilon$
does not alter the expression of $Q$. Further on, we can find a $\beta>0$ such that 
$$
  \| \theta \|_{L^p(\Omega \times (0,T_\varepsilon))} 
   \leq c T_\varepsilon^\beta \| \theta \|_{W^{s,2p}(0,T_\varepsilon;L^{p}(\Omega))}.
$$
Hence, we may conclude that for $R$ big enough, we can take $T_\varepsilon$ so small that
\begin{align*}
  \| \theta \|_{L^p(\Omega \times (0,T_\varepsilon))} 
  \leq{} cT_\varepsilon^\beta \| \theta \|_{W^{s,2p}(0,T_\varepsilon;L^{p}(\Omega))}
   \leq{} T_\varepsilon^\beta Q(  \| \varphi_0^\varepsilon\|_{\C^2(\ov\Omega)} , 
       \| \theta_0^\varepsilon \|_{\C^2(\ov\Omega)} , R) \leq R\,. 
\end{align*}

\smallskip
  
\textbf{Claim 2: The image of  $\mathcal{T}$ is compact in $\mathcal{M}_R$.}
The above estimate~\eqref{maxregtheta} allows us to apply the \textsc{Lions--Aubin} 
compactness lemma (\textit{cf}.~\cite[p.~58]{lions})
granting compactness in $L^{p/2}(\Omega\times (0,T))$. Combining this
with the boundedness~\eqref{estimthetapsq}, we obtain compactness in $L^{p}(\Omega\times (0,T))$. 
Note that we used $p>2$ such that $p^2>2p$. 

\smallskip
   
\textbf{Claim 3: The operator $\mathcal{T}$ is continuous.}
Assume that $\varphi_1$ and $\varphi_2$ are the solutions to~\eqref{app:phase}
corresponding to two different right-hand sides $\ov \theta_1 $ 
and $\ov \theta_2$, respectively. Subtracting the equations from
each other and testing the difference by $\varphi_1 - \varphi_2$, we find
\begin{align*}
 \frac{1}{2}\frac{\de }{\de t} \| \varphi_1 -\varphi_2 \|_{L^2(\Omega)}^2 
  + \| \nabla \varphi_1 - \nabla \varphi_2 \|_{L^2(\Omega)}^2 
  + \int_\Omega  \left ( G'(\varphi _1) -G'(\varphi_2) \right ) ( \varphi_1 - \varphi_2) \de \f x \\
  \leq  \frac{1}{2} \| \ov \theta_1 - \ov \theta _2 \|_{L^2(\Omega)}^2
   + \left ( \frac{1}{2} +2\lambda \right ) \| \varphi _1 - \varphi _2 \|_{L^2(\Omega)}^2 \,,
\end{align*}
which grants the continuity of the solution operator to the equation~\eqref{app:phase} 
from $L^2(\Omega\times (0,T))$ to $ L^\infty(0,T;L^2(\Omega))$. 

This fact, together with the previous estimates and the standard theory of the heat equation, guarantees 
that the solution operator to~\eqref{app:phase} maps $ \ov \theta$ to $\t\varphi$ continuously
from $L^2(\Omega \times (0,T)) $ to $L^2(\Omega \times (0,T))$. Since we already established 
the boundedness in a better space (see~\eqref{esttimeapp}), we even find that the solution operator
to~\eqref{app:phase} maps $ \ov \theta$ to $\t\varphi$ continuously from $L^p(\Omega \times (0,T)) $ 
to $ L^q (\Omega \times (0,T))$ for some $q<p$, for instance for $q = p/(p-2)$. 
  
Assume now that $\theta_1$ and $\theta_2$ are the solutions to~\eqref{app:heat}
corresponding to two different inputs $ \t\varphi_1 $ and $ \t\varphi_2$.
Subtracting the equations from each other and testing the difference by $\theta_1 - \theta_2$, 
we then find
\begin{align}\begin{split}
  \frac{1}{2}\frac{\de }{\de t}&\| \theta _1 - \theta_2 \|_{L^2(\Omega)}^2
   + \kappa \| \nabla \theta_1 - \theta_2 \|_{L^2(\Omega)}^2 
   + \varepsilon\int_\Omega ((\theta _1)^p - (\theta_2)^p )(\theta _1 - \theta_2 ) \de \f x \\
  = {}& \int_\Omega (| \t \varphi _1 |^2 -|\t \varphi_2|^2 
       -\t\varphi_1 \theta_1 + \t \varphi_2 \theta_2 )(\theta _1 - \theta _2 ) \de \f x \\
  \leq{}&  \| \t \varphi _1 - \t \varphi_2 \|_{L^{p/(p-2)}(\Omega)}
   \left ( \| \t \varphi_1\|_{L^p(\Omega)} 
       + \| \t \varphi_2\|_{L^p(\Omega)} \right ) \| \theta _1 - \theta _2 \| _{L^{p}(\Omega)}\\ 
  & + \| \t \varphi _1 \|_{L^p(\Omega)} \| \theta_1 - \theta_2 \|_{L^{2p/(p-1)}(\Omega)}^2 
   + \| \theta_2\|_{L^{p}(\Omega)} \| \t \varphi _1 -\t\varphi_2 \|_{L^{p/(p-2)}(\Omega)} 
     \| \theta_1 - \theta_2 \|_{L^{p}(\Omega)}\,.
\end{split}\label{uniqueest}
\end{align}
Since the function $s \mapsto (s)^p$ is strongly monotone, there exists a $c_p>0$ such that
\begin{align*}
 ((\theta _1)^p - (\theta_2)^p )(\theta _1 - \theta_2 ) 
  \geq c_p | \theta_1 - \theta_2 | ^{p+1}\,.
\end{align*}
This allows us to absorb the terms $\| \theta _1 - \theta _2 \| _{L^{p}(\Omega)} $ 
on the right-hand side of~\eqref{uniqueest} via \textsc{Young}'s inequality 
into the corresponding term on the left-hand side. 
The estimation of the other terms in \eqref{uniqueest} is simpler and hence 
it is left to the reader. 
Note that $p> 3$, so that $2p/(p-1)> p$.
Then, a number of straightforward checks permits us to conclude that the 
solution operator to~\eqref{app:heat} is continuous as a mapping from 
$L^{p/(p-2)}(0,T;L^{p/(p-2)}(\Omega))$ to $ L^\infty(0,T;L^2(\Omega))
\cap L^{p+1}(0,T;L^{p+1}(\Omega)) \cap L^2(0,T;H^1(\Omega))$.
Eventually, we have obtained that $\mathcal{T}$ is a continuous 
mapping from $\mathcal{M}_R$ to $\mathcal{M}_R$, as desired.

\smallskip

\textbf{Claim 4: The solution $\theta$ is positive.} We observe that 
from \eqref{app:heat} there follows
\begin{align*}
  \t \theta -\kappa \Delta \theta 
   & = - \varepsilon (\theta) ^p + | \t \varphi |^2 - \t \varphi \theta\\
   & \geq - \varepsilon (\theta)^p + \frac12 | \t \varphi |^2
  - \frac{1}{2}\theta ^2 
   \geq - (\theta)^p - \frac{1}{2}\theta ^2.
\end{align*}
Then, one may check that, for every $\varepsilon\in(0,1)$,
the solution to the ODE
\begin{align*}
  \t h + (h)^p + \frac{1}{2} h^2 = 0 \,, \quad h(0) 
  := \essinf_{\f x \in \Omega } \theta_0^\varepsilon (\f x)\,,
\end{align*}
is a positive subsolution to~\eqref{app:heat} on
$\Omega \times (0,T)$. Consequently, $\theta (\f x , t ) \geq h(t)>0$ 
for all $ \f x \in \Omega $ and $t\in (0,T)$.  
Moreover, it is worth observing that the above bound is, in the case of the additional assumption~\eqref{tetapos}, independent
of $\varepsilon$.  This fact is not essential here, but it will be in 
the sequel.

\smallskip

As an outcome of the proved Claims 1--4, we may now apply 
\textsc{Schauder}'s fixed point theorem to the map $\mathcal{T}$. 
Hence, for every $\varepsilon>0$ there exists a local-in-time solution 
$(\theta_\varepsilon , \varphi_\varepsilon)$ 
to the regularized system~\eqref{sys:reg}, 
where the local existence interval $(0,T_\varepsilon)$
may actually depend on $\varepsilon$.
\begin{remark}
 We have only established local-in-time existence of solutions to the approximate system. 
 On the other hand, all the \textit{a priori} estimates we shall perform in the next subsection
 will be independent of $T_\varepsilon$. Hence, as a consequence of standard
 extension arguments (details are left to the reader) the weak solution
 we obtain in the limit will be in fact defined on the whole reference interval
 $(0,T)$ given in the beginning.
\end{remark}


\subsection{\textit{A priori} estimates\label{sec:apri}}

In order to find \textit{a priori} bounds independent of $\varepsilon$, we derive a number of estimates.
Firstly, testing the equation~\eqref{sys:energy} and~\eqref{sys:phase} with 1 
and $\varphi_t$, respectively, and adding the results, we obtain the energy equality
\begin{multline}
  \frac{1}{2} \| \nabla \varphi_\varepsilon (t) \|_{L^2(\Omega)}^2 
   + \int_\Omega F(\varphi_\varepsilon(t)) + \theta_\varepsilon(t) \de \f x
  + \varepsilon \int_0^t  \int_\Omega \theta_\varepsilon ^p \de \f x \de s \\
  =  \frac{1}{2} \| \nabla \varphi^\varepsilon_0 \|_{L^2(\Omega)}^2 
   + \int_\Omega F(\varphi^\varepsilon_0) + \theta^\varepsilon_0  \de \f x \,,\label{eq:energy}
\end{multline}
where $\theta_0^\varepsilon$ and $\varphi_0^\varepsilon$ are the regularized initial data for
system~\eqref{app:sys}.

Similarly, testing~\eqref{sys:energy} by $-\vartheta/\theta_\varepsilon$, 
we find the entropy production rate
\begin{multline}
  -  \int_\Omega  ( \log \theta_\varepsilon (t) + \varphi_\varepsilon(t) ) \vartheta(t)  \de \f x 
   + \int_0^t \int_\Omega \kappa | \nabla \log \theta_\varepsilon|^2\vartheta 
   + \left | \frac{\t \varphi_\varepsilon}{\sqrt{\theta_\varepsilon}}\right | ^2\vartheta 
   - \varepsilon \theta_\varepsilon ^{p-1} \vartheta  \de \f x \de s \\ 
  = - \int_\Omega ( \log \theta^\varepsilon _0 + \varphi^\varepsilon_0 ) \vartheta(0) \de \f x 
  + \int_0^t  \int_\Omega \kappa \nabla \log \theta_\varepsilon  \cdot \nabla \vartheta 
  -  ( \log \theta_\varepsilon + \varphi_\varepsilon)\t \vartheta\de \f x \de s \,,  \label{eq:entro}
\end{multline}
holding true for every $\vartheta \in C^0([0,T]\times \overline\Omega) 
\cap H^1(0,T; L^2(\Omega)) \cap L^2 (0,T; H^1(\Omega))$ such that
$\vartheta(t, \f x)\ge 0$ for every $t\in[0,T]$ and $\f x \in \overline\Omega$.

Using the boundedness of $F$ from below (see Hypothesis~\ref{hypo}),
we conclude from the energy equality~\eqref{eq:energy} that
\begin{align}
  \|\nabla \varphi_\varepsilon \|_{L^\infty(0,T;L^2(\Omega)) }^2 
  + \| F(\varphi_\varepsilon) \|_{L^\infty(0,T;L^1(\Omega))} 
   + \| \theta _\varepsilon\|_{L^\infty(0,T;L^1(\Omega))} 
   + \varepsilon  \| \theta_\varepsilon \|_{L^p(0,T;L^p(\Omega))}^p  \leq c \,.\label{app:energy}
\end{align}
Then, on account of \eqref{Ecoerc}, we have more precisely
\begin{align}
  \|\varphi_\varepsilon \|_{L^\infty(0,T;H^1(\Omega)) } \leq c \,.\label{app:energy2}
\end{align}
From~\eqref{app:energy} and~Hypothesis~\ref{hypo}, we observe that
\begin{align}
  \| F'(\varphi_\varepsilon) \log (e + | F'(\varphi_\varepsilon) |) \|_{L^\infty(0,T;L^1(\Omega))}
  \leq c \,. \label{app:fderi}
\end{align}
Choosing now $\vartheta=1$ in~\eqref{eq:entro} and using~\eqref{app:energy},
we obtain the bounds
\begin{align}
  \| \log \theta_\varepsilon\|_{L^\infty(0,T;L^1(\Omega))} 
   + \| \nabla \log \theta_\varepsilon\|_{L^2(0,T;L^2(\Omega))} 
   + \left \| \frac{\t \varphi_\varepsilon }{\sqrt{\theta_\varepsilon}} \right \|_{L^2(0,T;L^2(\Omega)) }
 \leq c\,. \label{entroest} 
\end{align}
As observed in the introduction, the above information (which is the direct
outcome of the basic laws of Thermodynamics) is not sufficient to provide 
a control of the temperature uniformly with respect to $\varepsilon$
in a space better than $L^1$ with respect
to the space variables. This is a critical issue and it is the reason why
existence of weak solutions to system \eqref{eq1}-\eqref{eq2} has been proved
so far only in space dimension 1, or in the case when the heat flux law
has a power-like expression \cite{existence}, \textit{i.e.}, the diffusion term in 
\eqref{eq1} has the form $-\kappa \Delta \theta^b$, for $b$ large enough.
In what follows we actually deduce a new \textit{a-priori} estimate 
that will allow us to show that 
\begin{align}
  \int_0^T \int_\Omega \theta_\varepsilon \log \theta_\varepsilon \de \f x \de t \leq c \,.\label{Orlicz}
\end{align}
Such a uniform integrability property, which is 
a somehow unexpected fact because it does not seem to be linked to any 
physical principle, will be the key for showing 
that (a subsequence of) $\theta_\varepsilon$ converges strongly in $L^1$ 
to a limit temperature. In turn, this will imply solvability of the 
limit system in the sense of weak solutions with no
occurrence of defect measures. It is worth noting that, in order
for this argument to work, assuming a polynomial growth of the 
potential \eqref{growthF} is a crucial point. 

That said, we first observe that the solution to 
\begin{align}
  \t \un \varphi - \Delta \un \varphi + F'( \un \varphi ) = \t \un \varphi - \Delta \un \varphi
   + G'( \un \varphi ) - 2 \lambda \un \varphi = 0\quad \text{with }\un\varphi (0) = - K \label{subsol}
\end{align}
is a subsolution to~\eqref{sys:phase}, because we have proved that $ \theta \geq 0$. Note that we have chosen the 
initial value $\un\varphi (0)$ to be constant in $\Omega$, where the constant $K$ is given in~\eqref{in:chi2}. 

Then, assuming for simplicity but without loss of generality that $G'(0) = 0$,
for the solution of~\eqref{subsol}, we observe by testing with $ \un \varphi $ that 
\begin{align*} 
  \frac{1}{2}\left ( \frac{\de }{\de t } | \un \varphi | ^2
    - \Delta |  \un \varphi |^2 \right )
    = - | \nabla \un \varphi|^2 - G'( \un \varphi ) \un \varphi 
   + 2 \lambda |\un  \varphi |^2 \leq  2 \lambda |\un  \varphi |^2 \quad
   \text{with }|\un \varphi(0)|^2 = K^2\,,
\end{align*}
which
permits us to conclude that the solution to the ODE $ \t h = 4\lambda h $ with $h(0) = K^2$ is 
 a supersolution to 
$ | \un \varphi|^2 $  and deduce that $ | \un \varphi (\f x , t) |^2 
\leq  K ^2 e^{4\lambda T}$ 
for all $ (\f x ,t) \in \Omega \times (0,T)$. This in turn gives a lower 
bound on the solution of~\eqref{sys:phase}, \textit{i.e.}, 
there exists a constant $K'>0$ such that
\begin{align}
  \varphi_\varepsilon(\f x, t) \geq - K'
   \quad\text{for all $ (\f x ,t) \in \Omega \times (0,T)$}\,.\label{lowerbound}
\end{align}

Multiplying now~\eqref{sys:energy} by $ 1/\theta_\varepsilon$, we find the entropy production 
rate for the approximate system:
\begin{align*}
  \t \log \theta_\varepsilon 
   + \t \varphi_\varepsilon 
   - \kappa \Delta \log \theta_\varepsilon
   - \kappa | \nabla \log \theta_\varepsilon |^2 
   + \varepsilon \theta _\varepsilon^{p-1} 
  = \frac{| \t \varphi_\varepsilon | ^2 }{\theta_\varepsilon } \,.
\end{align*}
Testing this equation by $\varphi_\varepsilon+ K'$, we observe 
\begin{multline}
  \int_\Omega  (\varphi_\varepsilon+ K') \t \log \theta_\varepsilon \de \f x 
  + \frac{1}{2}\frac{\de }{\de t }\| \varphi_\varepsilon+ K' \|_{L^2(\Omega)}^2 \\ 
  + \int_\Omega \kappa \nabla \log \theta_\varepsilon \cdot \nabla \varphi_\varepsilon 
   - \kappa (\varphi_\varepsilon +K')| \nabla \log \theta_\varepsilon |^2 
   - (\varphi_\varepsilon +K')\frac{| \t \varphi_\varepsilon|^2}{\theta_\varepsilon } 
   + \varepsilon \theta_\varepsilon^{p-1}( \varphi _\varepsilon+K') \de \f x = 0 \,.\label{entrotestphi}
\end{multline}
Now, we test~\eqref{sys:phase} by $\log \theta_\varepsilon$, which implies
\begin{align}
  \int_\Omega\log \theta_\varepsilon \t  \varphi_\varepsilon 
  + \nabla \varphi_\varepsilon \cdot \nabla \log \theta _\varepsilon
   + F'(\varphi_\varepsilon ) \log \theta_\varepsilon \de \f x 
   = \int_\Omega \theta_\varepsilon \log \theta_\varepsilon \de \f x \,.\label{phasetestlog}
\end{align}
Adding~\eqref{entrotestphi} and~\eqref{phasetestlog} and integrating in time, we may observe
\begin{multline}
  \int_0^T\int_\Omega \theta_\varepsilon \log \theta_\varepsilon \de \f x \de t 
  +  \int_0^T \int_\Omega(\varphi _\varepsilon+K') \left (\kappa  | \nabla \log \theta_\varepsilon |^2 
  + \frac{| \t \varphi_\varepsilon|^2}{\theta_\varepsilon } \right )\de \f x \de t \\ 
  =  \frac{1}{2}\| \varphi_\varepsilon+K' \|_{L^2(\Omega)}^2 \Big|_0^T +  \int_\Omega  \log \theta_\varepsilon( \varphi_\varepsilon+K') \de \f x \Big |_0^T  
\\  + \int_0^T\int_\Omega ( \kappa + 1) \nabla \log \theta_\varepsilon \cdot \nabla \varphi _\varepsilon
  + F'(\varphi_\varepsilon ) \log \theta_\varepsilon   
   + \varepsilon \theta_\varepsilon^{p-1} (\varphi_\varepsilon+ K')  \de \f x\de t  \,.\label{together}
\end{multline}
Now, we provide a control of the various terms on the right-hand side of the above
relation. First of all, the lower bound on $\{\varphi_\varepsilon\}$ 
(see~\eqref{lowerbound}) allows us to observe that
\begin{align*}
  \int_0^t \int_\Omega ( \varphi_\varepsilon + K')\left (\kappa| \nabla \log \theta _\varepsilon|^2 
       +  \frac{| \t \varphi_\varepsilon|^2}{\theta_\varepsilon }\right )  \de \f x\de s 
 \geq 0\,.
\end{align*}
Moreover, the term~$\| \varphi_\varepsilon(t)\|_{L^2(\Omega)}^2$ 
in \eqref{together} is controlled due to the $L^\infty(0,T;H^1(\Omega))$-bound 
of $\{\varphi_\varepsilon\}$ established in~\eqref{app:energy2}. 
Next, using H\"older's inequality, we have
\begin{align*}
  \int_0^t\int_\Omega (\kappa + 1) \nabla \varphi_\varepsilon \cdot \nabla \log \theta_\varepsilon \de \f x \de t 
   \leq c \| \nabla \varphi_\varepsilon \|_{L^2(0,T;L^2(\Omega))} 
   \| \nabla \log\theta_\varepsilon \| _{L^2(0,T;L^2(\Omega))} \,,
\end{align*}
where the right-hand side is bounded due to~\eqref{app:energy} and~\eqref{entroest}. 
The regularizing term can be estimated as follows:
\begin{align*}
  \varepsilon \int_0^t \int_\Omega \theta ^{p-1}_\varepsilon \varphi_\varepsilon \de \f x \de t
   \leq \left ( \varepsilon^{1/p} \| \theta _\varepsilon \|_{L^p(0,T;L^p(\Omega))}\right )^{p-1} 
       \varepsilon^{1/p} \| \varphi_\varepsilon \|_{L^p(0,T;L^p(\Omega))} \le c \,,
\end{align*}
where we also used~\eqref{app:energy} and~\eqref{phiLp}.

Considering the second  term on the right-hand side of \eqref{together}, 
we distinguish the cases for the different signs of $\log \theta_\varepsilon$:
\begin{align}\begin{split}
  \int_\Omega  \log \theta_\varepsilon (\varphi_\varepsilon+ K') \de \f x 
   ={}& \int_{\{\theta_\varepsilon \leq 1 \}} \log \theta_\varepsilon (\varphi_\varepsilon+K') \de \f x
   + \int_{\{\theta_\varepsilon > 1 \}}    \log \theta _\varepsilon(\varphi _\varepsilon+K') \de \f x\,,\\ 
\end{split}\label{logphi}
\end{align}
where all terms are evaluated at the time $t$ (the counterparts
at $t=0$ are controlled thanks to the assumptions~\eqref{in:chi} 
and the assumptions on the regularized initial data below formula~\eqref{app:sys}). Thanks to the lower bound 
\eqref{lowerbound} on $\{\varphi_\varepsilon\}$, we find
\begin{align*}
 \int_{\{\theta_\varepsilon \leq 1 \}} \log \theta_\varepsilon (\varphi_\varepsilon+K') \de \f x \leq 0 
\,.
\end{align*}
%

Additionally, we may estimate the last 
term on the right-hand side of~\eqref{logphi}
using  the \textsc{Legendre--Fenchel--Young} inequality for the convex function $\psi(r) = e^r$ 
and its convex conjugate $ \psi^*(s) = s\log s - s $ (where we intend that $\psi^*(0)=0$ and 
$\psi^*(s)\equiv + \infty$ for $s<0$):
\begin{align*}
   \int_{\{\theta_\varepsilon > 1 \}
  }   
    \log \theta_\varepsilon (\varphi_\varepsilon+ K') \de \f x 
   & \leq  \int_{\{\theta_\varepsilon > 1 \}
   }  
   \psi(\log \theta_\varepsilon)+ \psi^*( \varphi_\varepsilon+K' ) \de \f x  \\
  & =  \int_{\{\theta _\varepsilon> 1 \}
  } 
      \theta_\varepsilon+( \varphi _\varepsilon+K')\log (\varphi_\varepsilon+K') - (\varphi _\varepsilon+K')\de \f x
  \leq c\,,
\end{align*}
the last inequality following from~\eqref{app:energy} and~\eqref{app:energy2}.

We now focus on the term depending on $F'$ 
in \eqref{together}. Actually, by \eqref{lowerbound} and \eqref{Fcoerc},
it is clear that $F'(\fhi_\varepsilon) \ge -c$ a.e.~in~$\Omega\times (0,T)$ and for all $\varepsilon>0$. 
Hence, we can decompose that term as 
\begin{align*}
  \int_0^t\int_\Omega F'(\varphi_\varepsilon ) \log \theta_\varepsilon \de \f x\de s
   & = \iint_{\{\theta _\varepsilon\leq 1 \}}  F'(\varphi_\varepsilon) \log \theta_\varepsilon  \de \f x\de s
   + \iint_{\{\theta_\varepsilon > 1 \}}   F'(\varphi_\varepsilon)  \log \theta_\varepsilon  \de \f x\de s \\
  \leq{}& c \iint_{\{\theta_\varepsilon \leq 1 \}} | \log \theta _\varepsilon |\de \f x\de s 
  + \iint_{\{\theta_\varepsilon > 1 \} }  
    | F'(\varphi_\varepsilon) | \log \theta_\varepsilon  \de \f x\de s \,.
\end{align*}
The first term on the right-hand side is estimated thanks to \eqref{entroest},
whereas the second term can be controlled once more by
using the \textsc{Legendre--Fenchel--Young} inequality with the same $\psi$ as
before. Namely, we have
\begin{align*}
 & \iint_{\{\theta_\varepsilon > 1 \} }  | F'(\varphi_\varepsilon) | \log \theta_\varepsilon  \de \f x\de s 
  \leq  \iint_{\{\theta_\varepsilon > 1 \}}  \psi^*( |F'(\varphi_\varepsilon)| ) + \psi( \log \theta_\varepsilon) \de \f x\de s \\
 & \mbox{}~~~~~ 
  \leq  \iint_{\{\theta_\varepsilon > 1 \}}   |F'(\varphi_\varepsilon)|(\log |F'(\varphi_\varepsilon)| - 1 ) +  \theta_\varepsilon \de \f x\de s  \,,
\end{align*}
and the right-hand side is controlled thanks to the growth assumption \eqref{growthF} 
of Hypothesis~\ref{hypo} and the bound~\eqref{app:energy}. 
Inserting everything back into~\eqref{together}, we eventually obtain
\begin{align*}
  \int_0^t\int_\Omega \theta_\varepsilon \log \theta_\varepsilon \de \f x \de s  \le c \,,
\end{align*}
which proves the estimate~\eqref{Orlicz}.

In order to improve our bound for $\t \fhi_\varepsilon$, we 
consider a new convex function $\psi:\RR\to[1,+\infty]$ defined as 
$\psi(r)=(1/4) (r^2(2\log r -1)+1)$, where it is intended that $\psi(1)=0$ and $\psi(r)\equiv +\infty$ as $r<1$. 
Determining the precise expression of the conjugate function $\psi^*(s)$ is difficult,
but we can at least estimate it appropriately. We recall
that 
\begin{equation*}
  \psi^*(s) = \max_{r\in \RR} \big( sr - \psi (r) \big)
\end{equation*}
and a simple computation shows that the maximum is attained at $s =r\log r$.
Hence, if $r$ is the maximizer, using first that $s =  r \log r$ 
and then that $s+1\le r^2$ (which holds as $r\geq 1$),
we have
\begin{align}\nonumber
  \psi^*(s) & =  r^2 \log r - \psi(r) = \frac{1}{2} r^2 \log r + \frac{1}{4} r^2 - \frac{1}{4} 
  = \frac{s^2}{\log r^2} + \frac{s^2}{\log^2 r^2} - \frac{1}{4}\nonumber  \\
  &\leq \frac{ s^2}{\log (s+1)} + \frac{s^2}{\log^2 (s+1) } - \frac{1}{4}\nonumber\,.
\end{align}
Additionally, we observe for $\psi^*$ that for any $y\in [2,\infty)$ it holds
\begin{align}
 \psi^*( y \log ^{1/2} y ) \leq \frac{y^2 \log  y}{\log\left  (1+y \log ^{1/2} y\right )}
 + \frac{y^2 \log  y}{\log^2\left  (1+y \log ^{1/2} y\right )} - \frac{1}{4} \leq c (y^2 + 1)\,,\label{squarebound}
\end{align}
since the function
\begin{align*}
  y \mapsto \frac{ \log  y}{\log\left  (1+y \log ^{1/2} y\right )}+ \frac{ \log  y}{\log^2\left  (1+y \log ^{1/2} y\right )} 
\end{align*}
is bounded for $y \in [2 ,\infty)$, which is obvious for any compact subset in $[2,\infty)$ 
and also holds as $y\nearrow \infty$ as an easy check shows. 

Now, setting for simplicity $u :=  \sqrt{\theta_\varepsilon}+1$ and  $v:=|\partial_t \fhi_\varepsilon |+2 u $,
we have 
\begin{align}\nonumber
  v \log^{1/2} v 
   & = \frac{v}{u}
     \bigg[ \log \Big( \frac{v}{u} \Big) + \frac12 \log(u^2 ) \bigg]^{1/2}u\\
 \nonumber
   & \le \frac{v}{u}
     \bigg[  \log^{1/2} \Big( \frac{v}{u} \Big) + \frac{1}{\sqrt 2}\log^{1/2}( u^2  ) \bigg] u  \\
 \nonumber
  & \le \frac{v}{u}
       \log^{1/2} \Big( \frac{v}{u} \Big)  u +  \frac{1}{\sqrt 2}\frac{v}{u} u \log^{1/2}( u^2  ) \\
\nonumber
  & \le \psi^*\left (        \frac{v}{u}
       \log^{1/2} \Big( \frac{v}{u} \Big)\right ) + \psi(u) + \frac{1}{\sqrt 8}\left ( \frac{v^2}{u^2} +  u^2 \log (u^2)   \right )    \\
  \nonumber
    &  \leq c \left (  \frac{v^2}{u^2}  +1\right ) + \frac{1}{4} u^2 ( 2 \log (u^2) -1 ) + \frac14 + \frac{1}{\sqrt 8}\left ( \frac{v^2}{u^2} +  u^2 \log (u^2) \right )\\
   &\leq c  \left (      \frac{v^2}{u^2}+u^2 \log (u^2)     +1 \right )
    \,,
 \label{conj14-2}
\end{align}
where we used calculation rules for the logarithm, properties of the square root under the additional observation that $\log (v/u)\geq 0$,
the \textsc{Legendre--Fenchel--Young} inequality as well as the standard \textsc{Young}'s inequality, and~\eqref{squarebound} as well as the definition of $\psi$. 
Finally, integrating \eqref{conj14-2} over $\Omega\times (0,T)$, 
we observe that the right-hand side is bounded due to \eqref{entroest} and \eqref{Orlicz}.
From the left-hand side, we deduce with the bound~\eqref{Orlicz} that
\begin{align}
  \int_0^T \int_\Omega | \partial_t \fhi_\varepsilon | \log^{1/2} \big( 1 + | \partial_t \fhi_\varepsilon | \big) \de \f x\de t 
   \le c.
   \label{boundphit}
\end{align}
Collecting the above estimates and comparing terms in the equality~\eqref{eq:entro}, we 
also obtain an estimate for the time derivatives of $\log \theta_\varepsilon$, \textit{i.e.},
\begin{align*}
\| \t \log \theta_\varepsilon \|_{L^1(0,T;(W^{1,p}(\Omega))^*)} \leq c \quad \text{for }p>3\,.
\end{align*}


\subsection{Convergence of the approximate solutions\label{sec:conv}}

The \textit{a priori} estimates deduced above allow us to infer that there exists a (nonrelabelled) subsequence
of $\varepsilon\searrow 0$ such that the following convergence relations hold:
\begin{align}
 \varphi_\varepsilon&\stackrel{*}{ \rightharpoonup} \varphi \quad \text{in }L^\infty(0,T;H^1(\Omega))\,,\label{weakconvphi}\\
 \log \theta_\varepsilon & \rightharpoonup \ell \quad \text{in } L^2(0,T;H^1(\Omega))\,,\label{weakconvlog}\\
 \t \log \theta_\varepsilon & \stackrel{*}{ \rightharpoonup} \t\ell  \quad \text{in } \mathcal{M}([0,T]; (W^{1,p}(\Omega))^* ) \text{ for }p>3 \,,\label{weakconvlog2}
\end{align}
where \eqref{weakconvlog} comes from combining the estimates of the first two quantities in \eqref{entroest}.

By estimate~\eqref{boundphit}, the family $\{\t \varphi _ \varepsilon\}$ is uniformly integrable. Hence, by
the \textsc{Dunford-Pettis} theorem (\textit{cf}.~\cite[Thms.~21, 22]{dm}) we have
\begin{align}
  \t \varphi _ \varepsilon &\rightharpoonup \t \varphi   \quad \text{in } L^1(0,T;L^1( \Omega)) \label{weak:tphi}\,.
\end{align}
Then, the generalized \textsc{Lions-Aubin} Lemma~(\textit{cf}.~\cite[Cor.~7.9]{roubicek} or~\cite[Thm.~3.19]{BV})  allows us to extract strongly converging subsequences:
\begin{align}
\varphi _\varepsilon & \ra \varphi \quad \text{in }L^2(0,T;L^2(\Omega))\,,\label{aefhi}\\
\log\theta_\varepsilon & \ra \ell \quad \text{in } L^2(0,T;L^2(\Omega))\,.
\end{align}
In particular, up to extracting further subsequences, $\{\fhi_\varepsilon\}$ and 
$\{\log\theta_\varepsilon\}$ converge to $\fhi$ and $\ell$ pointwise a.e.~in $\Omega\times (0,T)$. 
Using estimates~\eqref{app:fderi} and~\eqref{Orlicz}, we obtain that the families
$\{ F'(\varphi_\varepsilon)\}$ and $\{ \theta_\varepsilon\}$ are also uniformly integrable.
Hence, noting that $F'(\varphi_\varepsilon) \to F'(\varphi)$ 
pointwise a.e.~in $\Omega\times (0,T)$ by continuity of $F'$ and that $\theta_\varepsilon \to e^\ell=:\theta$ pointwise a.e.~in $\Omega\times (0,T)$,
and applying \textsc{Vitali}'s theorem (\textit{cf.}~\cite{vitali}), we obtain
\begin{align}
\theta_\varepsilon &\ra  \theta  \quad \text{in } L^1(0,T;L^1( \Omega)) \label{strongtheta} \,,\\
F'(\varphi_\varepsilon ) &\ra F'(\varphi )  \quad \text{in } L^1(0,T;L^1( \Omega )) 
\end{align}
and in particular we identify $\ell = \log\theta$
in \eqref{weakconvlog}-\eqref{weakconvlog2}.

Comparing terms in~\eqref{sys:phase}, we obtain that also $\{ \Delta \varphi_\varepsilon\}$
is uniformly integrable. Hence,
\begin{align}
  \Delta \varphi _\varepsilon & \rightharpoonup \Delta \varphi \quad \text{in } L^1(0,T;L^1( \Omega)) \,.
\end{align}
We can thus pass to the limit in all terms in~\eqref{sys:phase} and attain the equation~\eqref{phaseeq}.

In the next step, we want to pass to the limit in the energy equality~\eqref{eq:energy}.
First, we observe that $\varepsilon \theta^p\geq 0$. Hence, removing that term
we get an inequality at the approximate level.

From~\eqref{weakconvphi} and~\eqref{weak:tphi}
we deduce that 
\begin{align}
  \varphi  _\varepsilon \ra  \varphi \quad \text{in }\C_{\text{w}} ([0,T];H^1(\Omega))\,,\label{Cwconv}
\end{align}
 \textit{i.e.}, 
the sequence is converging pointwise in $[0,T]$ with respect to the weak topology in $L^2(\Omega)$. 
The weakly-lower semi-continuity of the $L^2(\Omega)$-norm 
grants that, for every $t\in[0,T]$,
\begin{align*}
  \| \nabla \varphi(t)\|_{L^2(\Omega)} ^2  \leq \liminf_{\varepsilon >0} \| \nabla \varphi_\varepsilon (t) \|_{L^2(\Omega)} ^2 \,.
\end{align*}
Finally, using pointwise (a.e.) convergence of $\{\log \theta_\varepsilon\}$ and $\{\varphi_\varepsilon\}$ as well as \textsc{Fatou}'s lemma, we find that
\begin{align*}
  \int_\Omega F(\varphi(t)) + \theta(t) \de \f x  \leq \liminf_{\varepsilon \ra 0} \int_\Omega F( \varphi_\varepsilon(t)) + \theta_\varepsilon(t) \de \f x \,\quad\text{a.e.~in }(0,T)\,.
\end{align*}
Therefore, the limit fulfills the energy inequality~\eqref{energyin}.

It remains to pass to the limit in the entropy inequality~\eqref{eq:entro} for $\vartheta\geq 0$. 
To remove the regularizing term, we observe that
\begin{align*}
  \int_0^t\int_\Omega \varepsilon \theta _\varepsilon^{p-1} \vartheta \de \f x \de t 
    \leq \varepsilon^{1/p} \left ( \varepsilon^{1/p} \| \theta _\varepsilon \|_{L^p(\Omega \times (0,T))} \right )^{p-1} \| \vartheta \|_{L^p(\Omega \times (0,T))} \ra 0 
\end{align*}
as $\varepsilon \ra 0$, where again the bound~\eqref{app:energy} is used. 
Additionally, combining estimate \eqref{entroest} with the pointwise convergence 
resulting from~\eqref{strongtheta} and~\eqref{weak:tphi},
we infer
\begin{align*}
\frac{\t \varphi_\varepsilon}{\sqrt{\theta_\varepsilon}} & \rightharpoonup \frac{\t \varphi }{\sqrt\theta}\quad \text{in }L^2(0,T;L^2(\Omega))\,.
\end{align*}
Considering now the equality~\eqref{eq:entro}, we find from the two previous convergences,~\eqref{weakconvlog} and the weak lower 
semi-continuity of convex functions (\textit{cf}.~\cite[Thm.~10.20]{fei} or~\cite{ioffe}) that
\begin{align}
\liminf_{\varepsilon\ra 0}  \int_0^t \int_\Omega \kappa | \nabla \log \theta_\varepsilon|^2\vartheta 
   + \left | \frac{\t \varphi_\varepsilon}{\sqrt{\theta_\varepsilon}}\right | ^2\vartheta 
   - \varepsilon \theta_\varepsilon ^{p-1} \vartheta  \de \f x \de s 
   \geq \int_0^t \int_\Omega \kappa | \nabla \log \theta|^2\vartheta 
   + \left | \frac{\t \varphi}{\sqrt{\theta}}\right | ^2\vartheta 
    \de \f x \de s\,,\label{ent1}
\end{align}
where, from now on, we consider $\vartheta \in C^0([0,T]\times \overline\Omega) 
\cap H^1(0,T; L^2(\Omega)) \cap L^2 (0,T; H^1(\Omega))$ such that
$\vartheta(t, \f x)\ge 0$ for every $t\in[0,T]$ and $\f x \in \overline\Omega$. 
Then, the convergence~\eqref{weakconvlog} allows us to pass to the limit in the second term on the right-hand side of~\eqref{eq:entro}:
\begin{multline}\label{ent2}
\lim_{\varepsilon\ra 0}  \int_0^t  \int_\Omega \kappa \nabla \log \theta_\varepsilon  \cdot \nabla \vartheta 
  -  ( \log \theta_\varepsilon + \varphi_\varepsilon)\t \vartheta\de \f x \de s \\
  = \int_0^t  \int_\Omega \kappa \nabla \log \theta  \cdot \nabla \vartheta 
  -  ( \log \theta + \varphi)\t \vartheta\de \f x \de s \quad\text{for all }t\in[0,T]\,.
\end{multline}
Due to the choice of the approximation of the initial values (below formula~\eqref{app:sys}) together with~\eqref{in:teta} as well as~\eqref{in:chi},
also the first term on the right-hand side of~\eqref{eq:entro} converges appropriately. 

From~\eqref{Cwconv}, we find that 
\begin{align}
\lim_{\varepsilon \ra 0} \int_\Omega \varphi_\varepsilon(t) \vartheta (t) \de\f x =  \int_\Omega \varphi (t)\vartheta(t) \de\f x \quad \text{for all~}t\in (0,T)
\,.\label{ent3}
\end{align}
Finally, we consider the first term in~\eqref{eq:entro}. Decomposing the logarithm in its positive and negative part,
\begin{align*}
\log \theta_\varepsilon = \log_+ \theta_\varepsilon - \log _- \theta_\varepsilon\,,
\end{align*}
we may observe that
the positive part is bounded in $L^p(\Omega)$ for any $p\in(0,\infty)$. 
Indeed, by the estimate $ y^q \leq c e^y $ for any $ q\geq 0$, we find
\begin{align}
  \int_\Omega | \log_+ \theta_\varepsilon(t)|^q \de \f x 
    \leq  \int_\Omega c e^{\log_+\theta_\varepsilon(t)}\de \f x 
    \leq c \int_\Omega \theta_\varepsilon(t)\de \f x \,\label{ent6}
\end{align}
for $q\in(0,\infty)$ and a.e.~$t\in(0,T)$, where the right-hand side is bounded due to~\eqref{app:energy}. 
This implies that the sequence $\{ \log_+\theta _\varepsilon (\f x ,t)\} $ is weakly compact in $L^1(\Omega)$ 
for a.e.~$t\in(0,T)$ (see for instance~\cite[Thm.~1.4.5]{RoubicekMeasure} for different characterizations of 
weak compactness in $L^1$) . The Theorem by \textsc{Vitali} (\textit{cf}.~\cite{vitali}) implies by
the pointwise (a.e.) strong convergence  of $\{ \log \theta_\varepsilon\}$ to $\log \theta$ that
\begin{align}
\log_+\theta_\varepsilon(t)& \ra \log_+\theta(t) \quad  \text{in } L^q(\Omega) \text{ for }q\in[1,\infty) \text{ and a.e.~}t\in(0,T)\,.\label{ent4}
\end{align}
From the a.e.~pointwise convergence of $\{ \log \theta_\varepsilon\}$ to $ \log \theta$ and \textsc{Fatou}'s Lemma, 
we deduce from the positivity of $\log_-$ and of $\vartheta$ that
\begin{align}
  \int_\Omega \log_-\theta (t) \vartheta(t) \de \f x
  \leq \liminf_{\varepsilon\searrow 0} \left (\int_\Omega \log_- \theta_\varepsilon(t) \vartheta(t) \de \f x \right )  \quad\text{a.e.~in }(0,T)\,.\label{ent5}
\end{align}
Finally, going to the limit in the entropy inequality~\eqref{eq:entro} for $\vartheta\geq 0$ 
using~\eqref{ent1}--\eqref{ent5}, we attain relation~\eqref{entropy}. 

Note that the pointwise lower bound on the approximate solutions of \textbf{Claim 4} in Sec.~\ref{sec:exreg} does not depend on $\varepsilon$ in the case of the additional assumption~\eqref{tetapos}. 
The strong convergence~\eqref{strongtheta} preserves this pointwise lower bound, which results in~\eqref{posifo}.
In this case, the sequence $ \{ \log \theta_\varepsilon \}$ is bounded in $L^\infty(0,T;L^q(\Omega))$ due to the lower bound and~\eqref{ent6}. 
This allows us to infer from convergence~\eqref{weakconvlog2} and~\cite[Lemma~2.17]{BV} that 
\begin{align*}
\log \theta_\varepsilon \stackrel{*}{\rightharpoonup} \log \theta \quad \text{in }\text{BV}(0,T;(W^{1,p}(\Omega))^*)\cap L^\infty(0,T;L^q(\Omega)) \text{ for }p>3 \text{ and }q\in (1,\infty)\,,
\end{align*}
which proves regularity \eqref{extra-log}. 
From \cite[Thm.~A.5]{roroweak} with the choices $Y=W^{1,p}(\Omega)$ and $V=L^q(\Omega)$ we also deduce that 
\[
  \log \theta_\varepsilon(t) {\rightharpoonup} \log \theta(t) \quad \text{in } L^q(\Omega)\quad \hbox{for }q\in (1,\infty) \hbox{ and a.e.~}t\in (0,T)\,.
\]
Moreover, using a {generalization} of \textsc{Helly}'s thorem (\textit{cf.}, \textit{e.g.}, \cite[Thm.~3.1]{mm} or \cite[Lemma~7.2]{ddm}) we also obtain 
\[
 \log \theta_\varepsilon(t)\stackrel{*}{ \rightharpoonup}\bar\ell(t) \quad \text{in } (W^{1,p}(\Omega))^*\quad \hbox{for all~}t\in (0,T).
\]
It is easy then to {check} that $\bar\ell$ coincides with $\log\theta$ almost everywhere. Hence, up to changing the representative of $\log\theta$, we may assume that
$\bar\ell=\log\theta$ everywhere on $[0,T]$. 

%



\section{Weak-strong uniqueness}
\label{sec:ws}

In this section we prove Theorem~\ref{thm:main}. In order to do that we first introduce a proper {\em relative energy functional}\ 
$\mathcal E$ (\textit{cf}.~\eqref{relen}), which plays the role of a {\em distance} between a solution $(\theta, \varphi)$ and a generic couple $(\tet, \tp)$. 
Then we estimate in a suitable way the various terms appearing in the expression of $\mathcal E$,
assuming that $(\tet, \tp)$ is a strong solution of our problem, with 
the aim of  finding a suitable {\em relative energy inequality} (\textit{cf}.~\eqref{firstest}). 
With proper manipulations, we will then see that the relative energy inequality
can be interpreted as a differential inequality to which a Gronwall-type argument (\textit{cf}.~\eqref{estimate})
can be applied. As a consequence of that, we will obtain that
any weak  solution necessarily coincides with
a strong solution originating from the same initial data on the existence interval of
the latter.


\subsection{Relative energy inequality}

We use the relative energy  approach (see \textsc{Feireisl}, \textsc{Jin} and \textsc{Novotn{\'y}}~\cite{Feireislrelative}) to prove weak-strong uniqueness. 
In the case of a convex energy functional, this idea goes back to \textsc{Dafermos}~\cite{dafermos} in the context of thermodynamical systems.

We define the relative energy for system~\eqref{eq} as
\begin{subequations}
\label{relen}
\begin{align}
  \mathcal{E}(\theta , \varphi  | \tet,\tp) : ={}&    \frac{1}{2}\| \nabla \varphi - \nabla \tp \|_{L^2(\Omega)}^2  
+ M \| \varphi - \tp \|_{L^1(\Omega)}^2 - \lambda\| \varphi - \tp \|_{L^2(\Omega)}^2  \label{relen0} \\
  & + \int_\Omega  G ( \varphi ) - G ( \tp ) - G'(\tp)( \varphi - \tp )  \de \f x+\int_\Omega  \Lambda ( \theta | \tet)  \de \f x  \label{relen1}     \,,
\end{align}
\end{subequations}
which plays the role of a {\em distance} between a solution $(\theta, \varphi)$ and a generic couple $(\tet, \tp)$. 
In \eqref{relen} we defined for convenience 
\begin{align}
 \Lambda ( \theta | \tet):= \theta - \tet - \tet ( \log \theta - \log \tet)\,.
\label{defLamb}
\end{align}
Moreover, the constant $M>0$ is taken big enough
so that, from the \textsc{Gagliardo-Nirenberg} inequality (\textit{cf}.~\cite[p.~125]{nier})
$ \| \phi\|_{L^2(\Omega)}\leq c (\| \nabla \phi\|_{L^2(\Omega)} ^\alpha \| \phi\|_{L^1(\Omega)}^{(1-\alpha) } + \| \phi\|_{L^1(\Omega)})$ 
for an $\alpha\in(0,1)$ and the \textsc{Young} inequality, line~\eqref{relen0} can be estimated from below by
\begin{equation}\label{82b}
 \mathcal{E}(\theta , \varphi  | \tet,\tp) 
   \geq {} \frac{1}{4}\| \nabla \varphi - \nabla  \tp \|_{L^2(\Omega)}^2 +\| \varphi - \tp\|_{L^1(\Omega)}^2    \,.
\end{equation}
Moreover, due to convexity of $G$, we may conclude that the first term in line~\eqref{relen1} is nonnegative.

We can also observe that, thanks to definition~\eqref{lambdacon},  
\begin{align*}
G ( \varphi ) - G ( \tp ) - G'(\tp)( \varphi - \tp )  = {}& F ( \varphi ) - F ( \tp ) - F'(\tp)( \varphi - \tp )  \\
&+ {\lambda}( | \varphi|^2 - | \tp |^2 - 2\tp ( \varphi - \tp )) \\
={}& F ( \varphi ) - F ( \tp ) - F'(\tp)( \varphi - \tp )  + {\lambda} | \varphi - \tp |^2 \,.
\end{align*}
Finally, we note that also the second term in \eqref{relen1} (\textit{cf.}~\eqref{defLamb}) is nonnegative thanks to convexity
of the exponential. 

Hence, regrouping of some terms in~\eqref{relen} gives
\begin{subequations}\label{calcrelen}
\begin{align}
\mathcal{E}(\theta , \varphi  | \tet,\tp) ={}& \int_{\Omega}  \frac{1}{2}| \nabla \varphi|^2 + F (\varphi) + \theta \de \f x
-\int_{\Omega}  \frac{1}{2}| \nabla \tp|^2 + F (\tp) + \tet \de \f  x\notag \\
&+ \int_{\Omega} \nabla \tp \cdot ( \nabla \tp - \nabla \varphi ) + F' ( \tp ) ( \tp - \varphi )\de\f x \label{phaseterms}\\ 
& -  \int_{\Omega} \tet ( \log  \theta + \varphi )  - \tet ( \log  \tet + \tp )  \de \f x\notag  \\ 
&+\int_{\Omega}\tet ( \varphi - \tp )     \de \f x + M\| \varphi - \tp \|_{L^1(\Omega)} ^2\,.  \label{mixedterms}
\end{align}
\end{subequations}
Then, we assume that $(\tet, \tp)$ is a strong solution of our problem with the aim of  finding a suitable {\em relative energy inequality}. 
First, we observe by the energy inequality~\eqref{energyin} for the weak solution and the energy equality (\eqref{energyin} with equality) for the strong solution that
\begin{multline}
 \int_{\Omega} \frac{1}{2}| \nabla \varphi(t)|^2 + F (\varphi(t)) + \theta(t) \de \f x  
   -\int_{\Omega}  \frac{1}{2}| \nabla \tp(t)|^2 + F (\tp(t)) + \tet(t) \de \f x  \\
\leq   \int_{\Omega} \frac{1}{2}| \nabla \varphi_0|^2 + F (\varphi_0) + \theta _0\de \f x  -\int_{\Omega} \frac{1}{2}| \nabla \tp_0|^2 + F (\tp_0) + \tet_0 \de \f x \,. \label{energyestimates}
\end{multline}
Note that the energy equality is valid for the strong solution since its regularity suffices 
to test the equations~\eqref{eq1} and ~\eqref{eq2} by $1$ and $\tp_t$, respectively. For brevity, we denote the time derivative by the subscript $t$. 

Next, we choose $\vartheta=\tet$ in~\eqref{entropy}. Note that this is possible on account
of Corollary~\ref{thm:local2} and Remark~\ref{rem:continuous}. 


We then get
\begin{multline}
-  \int_{\Omega} \tet  ( \log  \theta + \varphi) \de \f x  \Big|_0^t + \int_0^t \int_{\Omega} \tet  \left ( \kappa | \nabla\log  \theta|^2+ \frac{| \varphi_t|^2}{\theta }  \right ) \de \f x  \de t\\
  \leq \int_0^t \int_\Omega  \kappa   \nabla \log \theta \cdot \nabla \tet - \tet _t ( \log \theta + \varphi ) \de \f x  \de t \label{testedone}\,.
\end{multline}
Testing now equation~\eqref{eq1} for the strong solution $(\tet,\tp)$ by 1 and observing that
$ \tet_t+ \tp_t\tet= \t (\tet(\log\tet + \tp))-\tet_t ( \log\tet + \tp) $, we find
\begin{align*}
  \int_\Omega \tet ( \log \tet + \tp) \de \f x \Big|_0^t - \int_0^t \int_\Omega {| \tp_t|^2}\de \f x\de t  = \int_0^t \int_\Omega \tet_t ( \log \tet + \tp) \de \f x \de t \,.
\end{align*}
Note that the diffusive term of~\eqref{eq1} vanished due to the homogeneous \textsc{Neumann} boundary conditions (see~\eqref{boundary}).
Similarly, we find by testing equation~\eqref{eq1} for the strong solution $(\tet,\tp)$ by $(\theta - \tet) / \tet$ that 
\begin{align*}
 \int_0^t \int_{\Omega} (\theta - \tet)   \left ( \kappa  | \nabla \log \tet |^2+\kappa  \Delta \log \tet  + \frac{| \tp_t|^2}{\tet }  \right ) \de \f x  \de t
  = \int_0^t \int_\Omega( \theta  -  \tet  ) \t( \log \tet + \tp ) \de \f x  \de t\,.
\end{align*}
Adding the last two equations leads to
\begin{multline}
  \int_\Omega \tet ( \log \tet + \tp) \de \f x \Big|_0^t 
    +  \int_0^t \int_{\Omega} (\theta - \tet)   \left ( \kappa  | \nabla \log \tet |^2
        +\kappa  \Delta \log \tet  + \frac{| \tp_t|^2}{\tet }  \right )  - {| \tp_t|^2} \de \f x  \de t\\ 
  = \int_0^t \int_\Omega \tet_t ( \log \tet + \tp) \de \f x \de t +\int_0^t \int_\Omega( \theta  -  \tet  ) \t( \log \tet + \tp ) \de \f x  \de t \,.
 \label{testedtwo}
\end{multline}
Applying the fundamental theorem of calculus, we may infer for the terms in line~\eqref{phaseterms} that
\begin{align}\label{intphase}
\begin{split}
 \int_{\Omega} &\left ( \nabla \tp \cdot ( \nabla \tp - \nabla \varphi ) + F' ( \tp ) ( \tp - \varphi )\right ) \de \f x  \Big|_0^t\\ 
 = {}&\int_0 ^t \int_\Omega \left (  \nabla \tp _t \cdot (\nabla \tp - \nabla \varphi ) + F'' ( \tp ) \tp_t ( \tp - \varphi)   \right )\de \f x  \de s \\
 {}&
 - \int_0^t \int_\Omega( \tp_t - \varphi   _t )( \Delta  \tp  - F ' (\tp) ) \de \f x \de s 
\,.
 \end{split}
\end{align}
Note that the procedure leading to equality \eqref{intphase} is formal, but it could be easily
made rigorous by density arguments.

Next, writing~\eqref{eq2} both for the strong solution $(\tet,\tp)$ and for the weak solution 
$(\theta,\fhi)$ (indeed, the expression \eqref{phaseeq} in the definition of weak solutions
is equivalent), taking the difference, and testing it by $\t \tp$, 
we deduce
\begin{multline*}
  \int_0^t \int_\Omega     \nabla \tp_t  \cdot (\nabla \tp- \nabla \varphi ) \de \f x \de s \\
   =  \int_0^t \int_\Omega  (\tet-  \theta) \tp_t      -\tp_t (F' ( \tp) - F' ( \varphi ))  \de \f x \de s 
  - \int_0^t \int_\Omega(  \tp _t - \varphi _t) \tp_t \de \f x  \de s\,.
\end{multline*} 
Similarly, we find by testing equation~\eqref{eq2} for the strong solution~$(\tet,\tp)$ with $ (\tp_t- \varphi_t)$ that
\begin{align*}
- \int_0^t \int_\Omega (   \Delta \tp -  F' ( \tp)) (\tp_t - \varphi_t ) \de \f x \de s  
=   \int_0^t \int_\Omega \tet (\tp_t - \varphi_t)    - \tp _t (\tp _t - \varphi _t)   \de \f x \de s \,.
\end{align*} 
Inserting the latter two relations into~\eqref{intphase}, we may conclude that
\begin{align}
\begin{split}
 \int_{\Omega} &\left ( \nabla \tp \cdot ( \nabla \tp - \nabla \varphi ) + F' ( \tp ) ( \tp - \varphi )\right ) \de \f x  \Big|_0^t\\ 
 = {}&\int_0 ^t \int_\Omega \tp_t  \left (  F'' ( \tp ) ( \tp - \varphi) - F' ( \tp ) + F' ( \varphi)    \right )\de \f x  \de s  \\
 & +\int_0^t \int_\Omega (\tet-  \theta) \tp_t    \de \f x \de s
 - \int_0^t 2  \int_\Omega(  \tp _t - \varphi _t)\tp_t   - (\tp_t - \varphi   _t )   \tet  \de \f x  \de s \,.
\end{split}\label{testedthree}
\end{align}
To handle the first term in line~\eqref{mixedterms}, we apply again the fundamental theorem of calculus, leading to
\begin{align}
\begin{split}
\int_{\Omega}  \tet ( \varphi - \tp )   \de \f x  \Big|_0^t
={}&\int_0^t \int_{\Omega} \ \tet_t  ( \varphi - \tp )  + ( \varphi_t - \tp_t ) \tet     \de \f x \de s  
\,.
\end{split}\label{intmix}
\end{align}
Now, we insert~\eqref{energyestimates},~\eqref{testedone},~\eqref{testedtwo},~\eqref{testedthree}, and~\eqref{intmix} 
back into~\eqref{calcrelen}, which leads to the estimate
\begin{align}
\begin{split}
  \mathcal{E}(\theta (t), \varphi (t) | &\tet(t),\tp(t)) + \int_0^t \int_\Omega  \tet \frac{| \varphi_t|^2}{\theta} + ( \theta - \tet ) \frac{| \tp_t|^2 }{\tet	}  - | \tp_t|^2 + 
   2 (\tp_t-\varphi_t ) \tp_t   \de \f x \de s \\
  +\kappa \int_0^t \int_\Omega&   \tet    {| \nabla \log\theta|^2 } -    {\nabla \log\theta }\cdot \nabla \tet + ( \theta - 
  \tet )  {| \nabla\log \tet|^2} +    {\Delta\log \tet } (  \theta -  \tet)  \de \f x  \de s \\
   \leq{}& \mathcal{E}(\theta_0 , \varphi _0 | \tet_0,\tp_0 ) +\int_0 ^t \int_\Omega \tp_t  \left (  F'' ( \tp ) ( \tp - \varphi) - F' ( \tp ) + F' ( \varphi)    \right )\de \f x  \de s \\
& -\int_0^t \int_\Omega      \tet _t ( \log \theta + \varphi )   -  \tet _t  ( \log \tet + \tp )   \de \f x \de s\\
 &+\int_0^t \int_\Omega    ( \theta -  \tet )\partial_t ( \log \tet + \tp )  + (\tet -\theta ) \tp _t  \de \f x \de s\\
& + \int_0^t \int_{\Omega}  (  \tp_t - \varphi _t )\tet  + \tet_t  ( \varphi - \tp )  +  ( \varphi_t - \tp_t ) \tet   \de \f x\de s 
  + M \| \varphi (t)- \tp (t) \|_{L^1(\Omega)}^2 \,.
 \end{split}\label{firstest}
\end{align}
For the last three lines of the forgoing estimate, we observe due to several cancellations that
\begin{align*}
  -\int_\Omega \big (     \tet _t ( \log \theta + \varphi )   -  \tet _t  ( \log \tet + \tp ) 
    -&    ( \theta -  \tet )\partial_t ( \log \tet + \tp )  - (\tet -\theta ) \tp _t  \big )\de \f x \\
  +& \int_\Omega ( \tp_t - \varphi_t )\tet \de \f x 
         +\int_\Omega \tet_t  ( \varphi - \tp ) \de \f x 
         + \int_\Omega (\varphi_t  - \tp_t) \tet \de \f x  \\
  ={}& -\int_\Omega     \tet _t ( \log \theta + \varphi )   -  \tet _t  ( \log \tet + \tp )-\tet_t  ( \varphi - \tp ) \de \f x \\
  &+ \int_\Omega ( \theta -  \tet )\partial_t ( \log \tet + \tp )  + (\tet -\theta ) \tp _t  \de \f x \\
  ={}&   \int_\Omega \t( \log \tet  )( \theta - \tet- \tet   (\log \theta - \log \tet ))\de \f x  \\
  ={}& \int_\Omega (\t \log \tet) \Lambda ( \theta | \tet ) \de \f x  \,.
\end{align*}
Here and below, we are actually using notation~\eqref{defLamb}.


\subsection{Estimates for the dissipative terms and the nonconvex part}
\label{subsec:diss}

For the dissipative terms due to heat conduction, \textit{i.e.}, 
the terms in the second line of~\eqref{firstest}, we observe with some algebraic transformations that
\begin{align*}
  \tet &   {| \nabla \log\theta|^2 } -    {\nabla \log\theta }\cdot \nabla \tet + ( \theta - 
\tet )  {| \nabla\log \tet|^2} +    {\Delta\log \tet } (  \theta -  \tet) \\
 &=  \left (   \tet  \nabla \log \theta\cdot (  \nabla \log \theta- \nabla \log \tet ) + ( \theta - \tet ) | \nabla \log \tet |^2 + \Delta \log \tet   ( \theta - \tet )    \right )\\
 &= \tet \left ( | \nabla \log \theta - \nabla \log \tet |^2 + \nabla \log \tet \cdot(\nabla \log \theta - \nabla \log \tet ) \right ) \\ 
 &\quad + \left ( \theta - \tet - \tet ( \log \theta - \log \tet ) | \nabla \log \tet |^2  + \tet ( \log \theta - \log \tet ) | \nabla \log \tet|^2 \right )\\
& \quad +  \Delta \log \tet  ( \theta -\tet  - \tet ( \log \theta - \log \tet ))   + \Delta \log \tet    ( \tet ( \log \theta - \log \tet ) ) \,.
\end{align*}
From an integration-by-parts on the last term, using the fact that $ \nabla \log \tet \cdot \f n = ( \nabla \tet \cdot \f n ) / \tet = 0$ 
on the boundary (see~\eqref{boundary}),  and the product rule, we may infer
\begin{multline*}
 \int_\Omega \tet \nabla \log \tet \cdot(\nabla \log \theta - \nabla \log \tet ) + \tet ( \log \theta - \log \tet ) | \nabla \log \tet|^2\de \f x \\- \int_\Omega \nabla \log \tet \cdot  \nabla ( \tet ( \log \theta - \log \tet ) )\de \f x  = 0\,.
\end{multline*}
We may conclude that
\begin{align*}
  \int_\Omega  \tet &   {| \nabla \log\theta|^2 } -    {\nabla \log\theta }\cdot \nabla \tet
   + ( \theta - \tet )  {| \nabla\log \tet|^2} +    {\Delta\log \tet } (  \theta -  \tet) \de \f x    \\
  & = \int_{\Omega}  \tet | \nabla \log \theta - \nabla \log \tet |^2 \de \f x  
   + \int_{\Omega} \Lambda ( \theta | \tet)  ( | \nabla \log \tet|^2 + \Delta  \log \tet ) \de \f x  \,.
\end{align*}
Thanks to the above manipulations, the estimate~\eqref{firstest} may be written as
\begin{align}
\begin{split}
  \mathcal{E}&( \theta(t), \varphi(t)| \tet(t),\tp(t)) + \kappa  \int_0^t  \int_{\Omega}  \tet | \nabla \log \theta - \nabla \log \tet |^2  \de \f x \de s\\
  & + \int_0^t \int_\Omega  \tet \frac{| \varphi_t|^2}{\theta} + ( \theta - \tet ) \frac{| \tp_t|^2 }{\tet	}  - | \tp_t|^2     +2 ( \tp_t-\varphi_t ) \tp_t \de \f x \de s \\
   \leq{}& \mathcal{E}(\theta_0 , \varphi _0 | \tet_0,\tp_0 ) +\int_0 ^t \int_\Omega \tp_t  \left (  F'' ( \tp ) ( \tp - \varphi) - F' ( \tp ) + F' ( \varphi)    \right )\de \f x  \de s  \\
   & + \int_0^t \int_{\Omega}  (\t( \log \tet  ) -\kappa | \nabla \log \tet|^2 - \kappa  \Delta  \log \tet ) \Lambda ( \theta | \tet)   \de \f x  \de s\\
 & + M \| \varphi (t)- \tp (t) \|_{L^1(\Omega)}^2  \,.
\end{split}\label{estim}
\end{align}
For the terms in the second line of the right-hand side of the previous estimate, we find with equation~\eqref{eq1} for the strong solution $(\tet,\tp)$
\begin{align*}
  ( \log \tet  )_t -\kappa | \nabla \log \tet|^2 - \kappa  \Delta  \log \tet = \frac{\tet_t - \kappa \Delta \tet }{\tet}
  = \frac{| \tp_t|^2 - \tp _t \tet }{\tet }= \frac{ | \tp_t|^2 }{\tet} - \tp_t\,.
\end{align*}
Moreover, rearranging the terms in the second line of~\eqref{estim}, we obtain
\begin{align}
\begin{split}
& \int_0^t \int_\Omega   \tet \frac{| \varphi_t|^2}{\theta} + ( \theta - \tet ) \frac{| \tp_t|^2 }{\tet	}   - | \tp_t|^2   
 +2 ( \tp_t-\varphi_t ) \tp_t    \de \f x   
  \de s \\
  &= {} \int_0^t \int_\Omega   \tet \frac{| \varphi_t|^2}{\theta} + \theta  \frac{| \tp_t|^2 }{\tet	}  -2  \varphi_t  \tp_t  \de \f x  \de s \\
  &= {} \int_0^t \int_\Omega   \left |
  \sqrt{\frac{\tet}{\theta}} \varphi_t - \sqrt{\frac{\theta}{\tet}}\tp_t \right |^2     \de \f x  \de s     \,.
\end{split}\label{secdiss}
\end{align}


\begin{proposition}
Let $( \theta ,\varphi)$  be a weak solution according to Definition~\ref{def:weak} and let $(\tet,\tp)$ 
be a strong solution according to Theorem~\ref{thm:local}. 
Then it holds
\begin{align*}
  \| \varphi - \tp\|_{L^1(\Omega)}^2  \Big | _0^t \leq{}& ( 4\lambda+2) \int_0^t  \| \varphi - \tp\|_{L^1(\Omega)}^2 \de s
   +\int_0^t \|\Lambda ( \theta | \tet) \| _{L^1(\Omega) }^2 \de s \\ 
  & +  \int_0^t  \left (\int_\Omega \tet  | \log \theta - \log \tet  | \de \f x \right )^2 \de s \,,
\end{align*}
for all $t\in [0,T]$. 
\end{proposition}
\begin{proof}
The idea of the proof is to use $ \segn( \varphi - \tp)$ as a test function  in equation~\eqref{phaseeq} and subtract
equation~\eqref{eq2} for the strong solution equally tested with $ \segn( \varphi - \tp)$.
Since all terms in~\eqref{phaseeq} are elements of $L^1(\Omega \times (0,T))$, this procedure is allowed.
Setting $ \hat{\varphi} = \varphi - \tp$, we then find
\begin{align*}
 \int_\Omega \t\hat{\varphi}    \segn(\hat{\varphi}  ) - \Delta  \hat{\varphi} \segn(\hat{\varphi} ) + (F'( \varphi ) - F ' ( \tp) ) \segn( \hat{\varphi}  )
\de \f x  =  \int_\Omega ( \theta  - \tet) \segn (\hat{\varphi})  \de \f x\,.
\end{align*}
Moreover, it is clear that
\begin{align*}
 \t\hat{\varphi}    \segn(\hat{\varphi}  ) = \t | \hat{\varphi}| 
\end{align*}
and the $\lambda$-convexity of $F$ guarantees that 
\begin{align*}
  ( F'( \varphi ) - F ' ( \tp) ) \segn (\hat{\varphi} )  
   = (G ' ( \varphi ) - G'(\tp) )(\segn \hat{\varphi}) - 2 \lambda | \hat{\varphi}| 
    \geq - 2\lambda | \hat{\varphi}|  \,.
\end{align*}
Additionally, proceeding by approximation it is not difficult to show that
\begin{align*}
- \int_\Omega  \Delta  \hat{\varphi} \segn(\hat{\varphi} ) \de \f x \geq 0\,,
\end{align*}
where we point out that the boundary conditions~\eqref{boundary} of $\varphi$ are crucial 
for this argument. 

Collecting the above relations, observing additionally that $| \segn(\hat{\varphi})|\leq 1$, we find
\begin{align*}
  \t \int_\Omega | \hat{\varphi} | \de \f x \leq{}& 2\lambda  \int_\Omega | \hat{\varphi} | \de \f x + \int_\Omega | \theta - \tet| \de \f x \\
  \leq{}& 2 \lambda  \int_\Omega | \hat{\varphi} | \de \f x 
   + \int_\Omega \Lambda ( \theta  | \tet )  \de \f x +\int_\Omega \tet  | \log \theta - \log \tet  | \de \f x\,.
\end{align*}
Multiplying the above relation by $ \| \hat{\varphi}\|_{L^1(\Omega)}$
and applying \textsc{Young}'s inequality,  we may conclude that 
 \begin{align*}
\frac{\de}{\de t } \| \hat{\varphi}\|_{L^1(\Omega)} ^2   \leq {}&2 (2\lambda+1)   \| \hat{\varphi} \|_{L^1(\Omega)} ^2 
  +  \left (\int_\Omega  \Lambda ( \theta  | \tet)  \de \f x\right )^2 \\  
  &  + \left (\int_\Omega \tet  | \log \theta - \log \tet  | \de \f x \right )^2
\end{align*}
for a.e.~$t\in(0,T)$.
Integrating in time provides the assertion. 
\end{proof}

\smallskip

\begin{proposition}\label{lem:log:2}
 Let $\theta$ be a weak solution provided by Theorem~\eqref{thm:exweak} and $\tet$ a strong solution 
 according to Theorem~\ref{thm:local}, both originating from the same initial data 
 satisfying\/ \eqref{teta0}-\eqref{chi0}
 and\/ \eqref{teta1} and defined over the same time interval $(0,T)$.
 Then there exists a constant $c>0$ such that 
 \begin{equation}\label{stimalog}
   \| \log \theta - \log \tet\|_{L^1(\Omega)} ^2 \leq c \int_\Omega \Lambda ( \theta | \tet)  \de \f x \,,
 \end{equation}
 where $c$ only depends on the given data of the system.
\end{proposition}
\begin{proof}
In view of the second condition \eqref{teta0}, \eqref{posifo} is satisfied both by 
$\theta$ and by $\tet$. Setting $\eta = \log \theta$ and $\et = \log\tet$, using
Taylor's expansion, we can directly compute
 \begin{equation}\label{xi}
  \Lambda ( \theta | \tet)  
   = e^{\eta} - e^{\et} - e^{\et} (\eta - \et)
   = \frac12 e^{\xi} (\eta-\et)^2,
 \end{equation}
where the above formula holds at a.e.~point $(\f x,t)\in \Omega\times(0,T)$
and $\xi=\xi(\f x,t)$ lies between $\eta(\f x,t)$ and $\et(\f x,t)$.

In view of \eqref{posifo}, there exists $\delta>0$ depending only on the initial
data and on $T$ such that $\theta(\f x,t)\ge\delta$ and $\tet(\f x,t)\ge\delta$.
As a consequence, we also have $\xi\ge \log\delta$. Thus, rewriting \eqref{xi}
in terms of $\theta$ and $\tet$ and integrating in space, we readily obtain
\eqref{stimalog} (which is stated in term of the $L^1$-rather than
$L^2$-norm just for later convenience).
\end{proof}
%



\subsection{Estimate of the convex modification}

It remains to control the term resulting from the nonlinear potential $F$, 
\textit{i.e.}, the second term on the right-hand side of~\eqref{estim}:
\begin{align*}
  & \int_0^t \int_\Omega \tp_t (F'' ( \tp ) ( \tp - \varphi) - F' ( \tp ) + F' ( \varphi))  \de \f x  \de s\\
  & = \int_0^t \int_\Omega \tp_t (G'' ( \tp ) ( \tp - \varphi) - G' ( \tp ) + G' ( \varphi)) \de \f x  \de s,
\end{align*}
the equality holding as a consequence of \eqref{lambdacon}. We first observe that,
by \eqref{regochis} and standard Sobolev embeddings, 
there exists $\overline{\varphi}>0$ depending only on the data of the problem
such that $| \tp ( x ,t ) | \le \overline{\varphi}$ a.e.~in~$\Omega\times(0,T)$.
For simplicity, we set $\mathcal{R}( \varphi , \tp) : = G'' ( \tp ) ( \tp - \varphi) - G' ( \tp ) + G' ( \varphi)$
and, for $\ov{M}>\overline{\varphi}$ to be chosen below, we also define
$$
  \Omega_-(t):=\big\{x\in\Omega:~|\fhi(t,\f x)|\le \ov M\big\}
$$
and, correspondingly, $\Omega_+(t):=\Omega\setminus\Omega_-(t)$. It is then clear that
for a.e.~$t\in(0,T)$ and $\f x\in \Omega_-(t)$, there holds
\begin{align*}
  \big| \mathcal{R}( \varphi , \tp) \big| = {} & \bigg| G''( \tp  ) ( \tp - \varphi ) - \int_0^1 G''(\tp + (1-s)( \varphi -\tp) ) \de s ( \tp - \varphi ) \bigg|
 \\
 \leq{}&  \int_0^1 |G''( \tp  )-G''(\tp + (1-s)( \varphi -\tp) )| \de s|  \tp - \varphi | 
 \\
 \leq{}&  \int_0^1 c( \| F\|_{\mathcal{C}^{2,1}}, \overline{\varphi},\ov M) | \tp -\tp - (1-s)( \varphi -\tp) | \de s|  \tp - \varphi | 
 \\
 \leq{}& \frac 1 2 c( \| F\|_{\mathcal{C}^{2,1}},  \overline{\varphi},\ov M) | \varphi -\tp | ^2  
 \\
 =: {}& c | \tp - \varphi |^2,
\end{align*}
the last constant $c$ depending on $\ov M$, $\overline{\fhi}$ and the problem data. On the other hand, 
for $\f x\in \Omega_+(t)$, thanks to the growth condition \eqref{growthG}
of Hypothesis~\ref{hypo} we have
\begin{align*}
|\mathcal{R}( \varphi , \tp)| \leq {}& |G '' ( \tp)| | \tp - \varphi | + | G' ( \varphi )|+| G'( \tp) |\\
\leq{}&  c ( 1 + | \varphi- \tp|^2  ) + | G '(\varphi)| \\
\leq{}& c ( 1+|\varphi-\tp|^2 + G(\varphi) ) \,.
\end{align*}
As a consequence of the above argument, we deduce
\begin{align}\label{st81}
  & \int_0^t \int_\Omega | \tp_t | (G'' ( \tp ) ( \tp - \varphi) - G' ( \tp ) + G' ( \varphi)) \de \f x  \de s \\
 \nonumber
  & \le c \int_0^t \| \tp_t \|_{L^\infty(\Omega)} \left( \| \tp - \varphi \|_{L^2(\Omega)}^2
   + \int_{\Omega_+(s)} G(\varphi) \de \f x 
   + | \Omega_+(s) | \right) \de s\,,
\end{align}
where, recalling \eqref{82b} and using the \textsc{Gagliardo--Nirenberg} (\textit{cf}.~\cite[p.~125]{nier}) and \textsc{Young} inequalities,
the difference in the $L^2$-norm can be estimated by
\begin{align*}
  \| \varphi -\tp \|_{L^2(\Omega)}^2
    \leq c \left ( \| \nabla \varphi - \nabla \tp \|_{L^2(\Omega)}^2 + \| \varphi -\tp \|_{L^1(\Omega )}^2 \right ) 
    \leq c \mathcal{E}(\theta,\varphi | \tet ,\tp) \,.
\end{align*}
Let us also notice that, by Hypothesis~\ref{hypo} (see in particular \eqref{Fcoerc}), 
we can choose $\ov M$ so large that $G(r)\ge \lambda r^2/2$ for every $|r|\ge \ov M$. As a consequence,
for $\f x\in \Omega_+(t)$ there holds
\begin{align*}
  \frac12 G ( \varphi) - G( \tp) - G'(\tp)(\varphi-\tp) 
    & \ge \frac\lambda4 \fhi^2 - G( \tp) + G'(\tp) \tp - \frac\lambda{8} \fhi^2 - \frac2\lambda G'(\tp)^2\\
    & \ge \frac\lambda{8} \fhi^2 - c(\overline{\varphi}) \ge \frac\lambda{8} \ov M^2 - c(\overline{\varphi}) \ge K(\ov M,\overline{\varphi},\lambda),
\end{align*}
where the last inequality follows for suitable $K>0$ by possibly taking a larger 
value of $\ov M$ (still in a way that only depends on $\overline{\varphi}$,
hence on the fixed data of the problem). We eventually conclude that
\begin{align*}
 \mathcal{E}( \theta , \varphi | \tet , \tp)(s)
  \geq{} & \int_{\Omega_+(s)} \big( G ( \varphi) -G( \tp) -G'(\tp)(\varphi-\tp) \big) \de \f x \\
  \geq{} & \frac12 \int_{\Omega_+(s)} G(\fhi) \de \f x 
   + K | \Omega_+(s) |\,,
\end{align*}
so that \eqref{st81} gives
\begin{equation}\label{st82}
 \int_0^t \int_\Omega | \tp_t | (G'' ( \tp ) ( \tp - \varphi) - G' ( \tp ) + G' ( \varphi)) \de \f x  \de s
  \le c \int_0^t \| \tp_t \|_{L^\infty(\Omega)} \mathcal{E}( \theta , \varphi | \tet , \tp) \de s\,.
\end{equation}
Putting everything together and going back to relation~\eqref{estim}, we may observe
\begin{align*}
  \mathcal{E}&( \theta(t), \varphi(t)| \tet(t),\tp(t)) + \kappa \int_0^t  \int_{\Omega}  {\tet} | \nabla \log \theta - \nabla \log \tet |^2 + \left | \sqrt{\frac{\tet}{\theta}} \varphi_t - \sqrt{\frac{\theta}{\tet}}\tp_t \right |^2 \de \f x  \de s \\
 %
 %
  \leq{}& \mathcal{E}(\theta_0 , \varphi _0 | \tet_0,\tp_0 ) +c \int_0 ^t \| \tp_t\|_{L^\infty(\Omega)}  \mathcal{E}(\theta , \varphi | \tet, \tp ) +\left ( \left  \| \frac{|\tp_t|^2 }{\tet} \right \|_{L^\infty(\Omega)} + \| \tp_t\|_{L^\infty(\Omega)} \right )   \Lambda ( \theta | \tet)   \de s  \\
 %
 %
  & + M \int_0^t (4 \lambda +2) \| \varphi - \tp \|_{L^1(\Omega)} ^2 + \| \Lambda ( \theta | \tet) \|_{L^1(\Omega)}^2 
    + \| \tet \|_{L^\infty(\Omega)} \| \log \theta - \log \tet \|_{L^1(\Omega)}^2 \de s   \,.
\end{align*}
We point out that $ \mathcal{ E}(\theta , \varphi | \tet,\tp) $ is bounded in $L^\infty(0,T)$ by a constant only depending on the given data of the system. 
Moreover, by \eqref{regotetas2}, $\| \tet \|_{L^\infty(\Omega)}$ is controlled uniformly in time. This in turn also holds for 
$ \| \Lambda (\theta | \tet) \|_{L^1(\Omega)}$. 
Hence we may arrive with Proposition~\ref{lem:log:2} at
\begin{align*}
  \mathcal{E}&( \theta(t), \varphi(t)| \tet(t),\tp(t)) + \int_0^t\mathcal{W}( \theta ,\varphi | \tet ,\tp)  \de s   
    \leq \mathcal{E}(\theta_0 , \varphi _0 | \tet_0,\tp_0 ) + c \int_0^t \mathcal{K}(\tet,\tp) \mathcal{E}( \theta ,\varphi | \tet ,\tp) \de s \,,
\end{align*}
where we have defined
\begin{align*}
  \mathcal{W}( \theta ,\varphi | \tet ,\tp) = {}&  \int_{\Omega} \left (\kappa  {\tet} | \nabla \log \theta - \nabla \log \tet |^2 
   + \left | \sqrt{\frac{\tet}{\theta}} \varphi_t - \sqrt{\frac{\theta}{\tet}}\tp_t \right |^2 \right )  \de \f x  
              \intertext{and}
    \mathcal{K} (\tet,\tp) ={}& c\left ( \| \tp_t(s)\|_{L^\infty(\Omega)} + \left \| \frac{|\tp_t(s)|^2}{\tet(s)}\right \|_{L^\infty(\Omega)} +1 \right  )  \,.
\end{align*}
Applying \textsc{Gronwall}’s inequality, we conclude that
\begin{multline}
  \mathcal{E}( \theta(t), \varphi(t)| \tet(t),\tp(t)) +\int_0^t \mathcal{W}( \theta(s) ,\varphi(s) | \tet(s) ,\tp(s)) e^{\int_s^t\mathcal{K}(\tet(\tau),\tp(\tau))\de \tau } \de s  \\ 
  \leq \mathcal{E}(\theta_0 , \varphi _0 | \tet_0,\tp_0 )   e^{\int_0^t\mathcal{K}(\tet(s),\tp(s))\de s } \,.\label{estimate}
\end{multline}
The above estimate concludes the proof of Theorem~\ref{thm:main}.
\begin{remark}
  It is worth noting that the previous estimate~\eqref{estimate} 
  could also be adapted to provide a result on the continuous dependence of 
  the weak solution on the initial datum, holding as long as a strong solution exists. 
\end{remark}

\begin{remark}\label{rem:dissalt}
In the case when the polynomial growth condition \eqref{growthF} fails but
$F$ has at most exponential growth (in such a way that $|F'|$ is 
somehow controlled by $|F|$, so excluding the so-called ``singular'' potentials), 
it may be still possible to prove existence of 
some notion of weak solution. However one expects the occurrence of defect
measures in the phase field equation because the uniform integrability
estimates for $\theta$, $\fhi_t$ and $F'(\fhi)$ are now expected to 
fail. On the other hand, strong positivity \eqref{posifo} 
of $\theta$ is still likely holding under 
the additional assumption \eqref{tetapos}, but it
should be reinterpreted in the sense of measures. Note, however,
that in such a setting, the interpretation of the term $\fhi_t^2/\theta$ 
in the (weak analogue of the) entropy 
inequality~\eqref{entropy} may be troublesome because both $\fhi_t$ and
$\theta$ are now measure-valued objects. Extending weak-strong uniqueness to this weakened 
regularity framework seems also a nontrivial issue; indeed, with the occurrence
of defect measures, several terms we could treat by taking advantage
of $L^1$-regularity may become difficult to be managed.
\end{remark}
%




\section{Local strong solutions}
\label{sec:local}


In what follows, we focus on the proof of Theorem~\ref{thm:local}.



To prove local solvability, we derive some local-in-time estimates.
In order to avoid technical complications, we work directly on system 
\eqref{eq1}-\eqref{eq2}. It is however easy to check that the argument
could be reproduced and made fully rigorous by working on the regularized 
system \eqref{sys:reg}. In the sequel, for notational simplicity, we will
write, for instance, $\| \cdot \|_{L^2}$ in place of 
$\| \cdot \|_{L^2(\Omega)}$. 

Differentiating~\eqref{eq2} in time leads to
\begin{align}
  \varphi_{tt}-\Delta \varphi_t + F''(\varphi) \varphi_t = \theta_t\,. \label{secondorder}
\end{align}
Testing~\eqref{secondorder} by $-\Delta \varphi_t+\varphi_t$ provides
\begin{align}
\begin{split}
 & \frac{\de}{\de t} \frac{1}{2} \|  \varphi_t\|_{H^1}^2 
   + \|  \Delta  \varphi_t\|_{L^2}^2
   + \| \nabla \varphi_t \|_{L^2}^2  
  = \int_{\Omega} F '' ( \varphi ) \varphi_t( \Delta \varphi _t-\varphi_t) \de \f x
   -  \int_{\Omega} \theta_t (\Delta\varphi_t -\varphi_t) \de \f x  \\
  & \mbox{}~~~~~~~~~~ = - \int_{\Omega} F '' ( \varphi )\left  (|\nabla \varphi_t|^2+| \varphi_t|^2\right ) 
  + F'''(\varphi)  \varphi _t\nabla \varphi \cdot  \nabla  \varphi _t \de \f x 
   - \int_{\Omega} \theta_t (\Delta\varphi_t - \varphi_t)\de \f x\\
  & \mbox{}~~~~~~~~~~ \le \lambda \|  \varphi_t\|_{H^1}^2
   + \int_\Omega F'''(\varphi)  \varphi _t\nabla \varphi \cdot  \nabla  \varphi _t \de \f x
   + \frac{9}{16} \| \theta_t\|_{L^2}^2 + \frac{1}{2} \| \Delta \varphi_t\|_{L^2}^2+ 4\| \varphi_t\|_{L^2}^2  \,.
\end{split}\label{localest1}
\end{align}
In the above formula we used the $\lambda$-convexity of $F$ (see Hypothesis~\ref{hypo}) 
together with \textsc{H\"older}'s and \textsc{Young}'s inequalities.

Testing now equation~\eqref{eq1} with $\theta _t$, we obtain
\begin{align}
\begin{split}
\frac{\de }{\de t}\frac{\kappa}{2} \|\nabla \theta \|_{L^2}^2 + \|  \theta_t  \|_{L^2} ^2
  ={}&  \int_\Omega | \varphi_t | ^2 \theta _t - \theta \theta_t \varphi_t \de \f x  
\\
\leq{}&
\| \varphi_t\|_{L^4}^2 \| \theta_t \|_{L^2} + \| \theta_t\|_{L^2} \| \theta \|_{L^4} \| \varphi_t\|_{L^4}
\\
\leq {}& \frac{1}{8}\| \theta_t\|_{L^2}^2 +6 \| \varphi_t\|_{L^4}^4 + 2\| \theta\|_{L^4}^4 
\,,
\end{split}\label{localest2}
\end{align}
where we used again \textsc{H\"older}'s and \textsc{Young}'s inequalities.
With the chain rule, \textsc{H\"older}'s and \textsc{Young}'s inequalities, 
we observe the estimates
\begin{align}
 \begin{split}
   \frac{\de }{\de t}\frac{1}{2}\| \Delta \varphi\|_{L^2}^2 
  & = \int_\Omega \Delta \varphi _t \Delta \varphi \de\f x \leq \frac{1}{4}\| \Delta \varphi_t\|_{L^2} ^2 
    + \| \Delta \varphi \|_{L^2}^2 \,,\\
 %
 %
  \frac{\de }{\de t}\frac{1}{2}\|  \varphi\|_{L^2}^2 
   &= \int_\Omega  \varphi _t  \varphi \de\f x \leq \frac{1}{2}\|  \varphi_t\|_{L^2} ^2 
   +\frac{1}{2} \|  \varphi \|_{L^2}^2 \,,\\
  \frac{\de }{\de t}\frac{\kappa}{2}\|  \theta\|_{L^2}^2 
    & = \kappa \int_\Omega  \theta _t  \theta \de\f x \leq \frac{1}{16}\|  \theta_t\|_{L^2} ^2 
   + 4\kappa^2 \|  \theta \|_{L^2}^2\,.
 \end{split}
 \label{localtest3}
\end{align}
Adding~\eqref{localest1},~\eqref{localest2}, and~\eqref{localtest3} and using on $H^2_{\f n}$ 
the norm $\| \cdot \|_{\tilde H^2}^2:= \| \cdot \|_{L^2}^2 + \| \Delta \cdot \|_{L^2}^2$,
which is equivalent to the standard $H^2$-norm as far as functions in  $H^2_{\f n}$  are considered, 
we obtain the inequality
\begin{align}
 \begin{split}
  & \frac{\de}{\de t} \frac{1}{2}\left (  \|  \varphi_t\|_{H^1}^2 +{\kappa} \| \theta \|_{H^1}^2
   + \| \varphi\|_{\tilde H^2}^2\right )  +\frac{1}{4}\left ( \| \Delta  \varphi_t\|_{L^2}^2 
    +\|  \theta_t  \|_{L^2} ^2\right )\\
  & \mbox{}~~~~~ \leq {}c\left ( \|  \varphi_t\|_{H^1}^2+ \| \varphi\|_{\tilde H^2}^2 + \| \theta\|_{L^2}^2\right )
   +\int_\Omega F'''(\varphi)  \varphi _t\nabla \varphi \cdot  \nabla  \varphi _t \de \f x   
   + 6 \| \varphi_t\|_{L^4}^4 + 2\| \theta\|_{L^4}^4 \,.
\end{split} \label{loc:est}
\end{align}
For the term including the nonconvex potential, we observe that
\begin{align*}
 \int_\Omega F'''(\varphi)  \varphi _t\nabla \varphi \cdot  \nabla  \varphi _t \de \f x  
   \leq {}& \| \nabla \varphi _t \|_{L^2} \| \nabla \varphi\|_{L^3}  \| \varphi_t\|_{L^6}\|  F'''(\varphi) \|_{L^\infty}\\
   \leq{}&c \| \varphi_t\|_{H^1}^2\| \nabla \varphi\|_{L^3} 
    \left (\max_{s\in[-\| \varphi\|_{L^\infty},\| \varphi\|_{L^\infty}]}|F'''(s)|\right ) \\
   \leq{}&c  \| \varphi_t\|_{H^1}^2 Q( \| \varphi \| _{\tilde H^2} ).
\end{align*}
Here and below, $Q:[0,\infty)\to[0,\infty)$ denotes a computable,
continuous and increasingly monotone function whose expression may
vary on occurrence. Here, in the specific, we used the condition 
that $F'''$ is continuous (see Hypothesis~\ref{hypo2})
and the continuity of the embedding of $H^2$ into $L^\infty$. 

Defining now $\xi (t) = \frac{1}{2}\big(   \|  \varphi_t\|_{H^1}^2 
 +{\kappa} \| \theta \|_{H^1}^2+ \|  \varphi\|_{\tilde H^2}^2\big) $, 
we find with the continuous embedding $ H^1(\Omega) \hookrightarrow L^4(\Omega)$
that the inequality
\begin{equation}\label{diff:ineq}
  \frac{\de }{\de t} \xi (t) \leq c\left (1+  Q({\xi(t)}) \right ) \,
\end{equation}
holds for $ Q:[0,\infty)\ra [0,\infty)$ with the properties specified above.
Then, a simple application of the comparison principle for ODE's
guarantees the existence of $T^*>0$ and $C_0>0$ such that
\begin{equation}\label{xi:bound}
  \| \xi \|_{L^\infty(0,T^*)}\le C_0. 
\end{equation}
We used here conditions \eqref{teta0}-\eqref{chi0} on the initial data. 
Indeed, it is not difficult to verify that the finiteness of the initial value $\xi|_{t=0}$ corresponds exactly
to the regularity of $\theta_0$, $\fhi_0$ and $\fhi_1=\fhi_t(0)$ specified 
in~\eqref{teta0}-\eqref{chi0}. Then, using \eqref{xi:bound} and subsequently
integrating \eqref{loc:est} over $(0,T^*)$, we deduce the 
properties \eqref{regotetas}-\eqref{regochis},
with the exception of the regularity condition $\theta \in L^2(0,T^*;H^2(\Omega))$.
The latter, however, can be inferred {\it a posteriori} by comparing terms 
in \eqref{eq1} and applying standard elliptic regularity 
results.
Note, finally, that the energy equality \eqref{energyin} is valid for the
strong solution with equality sign because  its regularity suffices 
to test the equations~\eqref{eq1} and ~\eqref{eq2} by $1$ and $\varphi_t$, respectively.
This concludes the proof of Theorem~\ref{thm:local}.

\medskip

\section*{Acknowledgements}
This research was supported by the Italian Ministry of Education, University and Research (MIUR): Dipartimenti di Eccellenza Program (2018--2022) -- Dept. of Mathematics ``F. Casorati'', University of Pavia. In addition, it has been performed in the framework of the project
Fondazione Cariplo-Regione Lombardia MEGASTAR ``Matematica
d'Eccellenza in biologia ed ingegneria come acceleratore di una nuova
strateGia per l'ATtRattivit\`{a} dell'ateneo pavese''. The paper also benefits from the support of the
GNAMPA (Gruppo Nazionale per l'Analisi Matematica, la Probabilit\`{a} e le
loro Applicazioni) of INdAM (Istituto Nazionale di Alta Matematica) for ER and GS.






\bibliographystyle{abbrv}

\end{document}